%% file: main_2_.tex
\documentclass[12pt]{amsart}

\usepackage[margin=1in,marginparwidth=0.8in, marginparsep=0.1in]{geometry}

\pdfoutput=1

\usepackage[draft]{say}
\definecolor{egyptianblue}{rgb}{0.06, 0.2, 0.65}

\usepackage[dvipsnames]{xcolor}
\definecolor{MHcol}{RGB}{16, 195, 235}

\usepackage{etoolbox}

\makeatletter
\let\old@tocline\@tocline
\let\section@tocline\@tocline
\newcommand{\subsection@dotsep}{4.5}
\newcommand{\subsubsection@dotsep}{4.5}
\patchcmd{\@tocline}
  {\hfil}
  {\nobreak
     \leaders\hbox{$\m@th
        \mkern \subsection@dotsep mu\hbox{.}\mkern \subsection@dotsep mu$}\hfill
     \nobreak}{}{}
\let\subsection@tocline\@tocline
\let\@tocline\old@tocline

\patchcmd{\@tocline}
  {\hfil}
  {\nobreak
     \leaders\hbox{$\m@th
        \mkern \subsubsection@dotsep mu\hbox{.}\mkern \subsubsection@dotsep mu$}\hfill
     \nobreak}{}{}
\let\subsubsection@tocline\@tocline
\let\@tocline\old@tocline

\let\old@l@subsection\l@subsection
\let\old@l@subsubsection\l@subsubsection

\def\@tocwriteb#1#2#3{%
  \begingroup
    \@xp\def\csname #2@tocline\endcsname##1##2##3##4##5##6{%
      \ifnum##1>\c@tocdepth
      \else \sbox\z@{##5\let\indentlabel\@tochangmeasure##6}\fi}%
    \csname l@#2\endcsname{#1{\csname#2name\endcsname}{\@secnumber}{}}%
  \endgroup
  \addcontentsline{toc}{#2}%
    {\protect#1{\csname#2name\endcsname}{\@secnumber}{#3}}}%

\newlength{\@tocsectionindent}
\newlength{\@tocsubsectionindent}
\newlength{\@tocsubsubsectionindent}
\newlength{\@tocsectionnumwidth}
\newlength{\@tocsubsectionnumwidth}
\newlength{\@tocsubsubsectionnumwidth}
\newcommand{\settocsectionnumwidth}[1]{\setlength{\@tocsectionnumwidth}{#1}}
\newcommand{\settocsubsectionnumwidth}[1]{\setlength{\@tocsubsectionnumwidth}{#1}}
\newcommand{\settocsubsubsectionnumwidth}[1]{\setlength{\@tocsubsubsectionnumwidth}{#1}}
\newcommand{\settocsectionindent}[1]{\setlength{\@tocsectionindent}{#1}}
\newcommand{\settocsubsectionindent}[1]{\setlength{\@tocsubsectionindent}{#1}}
\newcommand{\settocsubsubsectionindent}[1]{\setlength{\@tocsubsubsectionindent}{#1}}

\renewcommand{\l@section}{\section@tocline{1}{\@tocsectionvskip}{\@tocsectionindent}{}{\@tocsectionformat}}%
\renewcommand{\l@subsection}{\subsection@tocline{2}{\@tocsubsectionvskip}{\@tocsubsectionindent}{}{\@tocsubsectionformat}}%
\renewcommand{\l@subsubsection}{\subsubsection@tocline{3}{\@tocsubsubsectionvskip}{\@tocsubsubsectionindent}{}{\@tocsubsubsectionformat}}%
\newcommand{\@tocsectionformat}{}
\newcommand{\@tocsubsectionformat}{}
\newcommand{\@tocsubsubsectionformat}{}
\expandafter\def\csname toc@1format\endcsname{\@tocsectionformat}
\expandafter\def\csname toc@2format\endcsname{\@tocsubsectionformat}
\expandafter\def\csname toc@3format\endcsname{\@tocsubsubsectionformat}
\newcommand{\settocsectionformat}[1]{\renewcommand{\@tocsectionformat}{#1}}
\newcommand{\settocsubsectionformat}[1]{\renewcommand{\@tocsubsectionformat}{#1}}
\newcommand{\settocsubsubsectionformat}[1]{\renewcommand{\@tocsubsubsectionformat}{#1}}
\newlength{\@tocsectionvskip}
\newcommand{\settocsectionvskip}[1]{\setlength{\@tocsectionvskip}{#1}}
\newlength{\@tocsubsectionvskip}
\newcommand{\settocsubsectionvskip}[1]{\setlength{\@tocsubsectionvskip}{#1}}
\newlength{\@tocsubsubsectionvskip}
\newcommand{\settocsubsubsectionvskip}[1]{\setlength{\@tocsubsubsectionvskip}{#1}}

\patchcmd{\tocsection}{\indentlabel}{\makebox[\@tocsectionnumwidth][l]}{}{}
\patchcmd{\tocsubsection}{\indentlabel}{\makebox[\@tocsubsectionnumwidth][l]}{}{}
\patchcmd{\tocsubsubsection}{\indentlabel}{\makebox[\@tocsubsubsectionnumwidth][l]}{}{}

\newcommand{\@sectypepnumformat}{}
\renewcommand{\contentsline}[1]{%
  \expandafter\let\expandafter\@sectypepnumformat\csname @toc#1pnumformat\endcsname%
  \csname l@#1\endcsname}
\newcommand{\@tocsectionpnumformat}{}
\newcommand{\@tocsubsectionpnumformat}{}
\newcommand{\@tocsubsubsectionpnumformat}{}
\newcommand{\setsectionpnumformat}[1]{\renewcommand{\@tocsectionpnumformat}{#1}}
\newcommand{\setsubsectionpnumformat}[1]{\renewcommand{\@tocsubsectionpnumformat}{#1}}
\newcommand{\setsubsubsectionpnumformat}[1]{\renewcommand{\@tocsubsubsectionpnumformat}{#1}}
\renewcommand{\@tocpagenum}[1]{%
  \hfill {\mdseries\@sectypepnumformat #1}}

\let\oldappendix\appendix
\renewcommand{\appendix}{%
  \leavevmode\oldappendix%
  \addtocontents{toc}{%
    \protect\settowidth{\protect\@tocsectionnumwidth}{\protect\@tocsectionformat\sectionname\space}%
    \protect\addtolength{\protect\@tocsectionnumwidth}{2em}}%
}
\makeatother

\makeatletter
\settocsectionnumwidth{2em}
\settocsubsectionnumwidth{2.5em}
\settocsubsubsectionnumwidth{3em}
\settocsectionindent{1pc}%
\settocsubsectionindent{\dimexpr\@tocsectionindent+\@tocsectionnumwidth}%
\settocsubsubsectionindent{\dimexpr\@tocsubsectionindent+\@tocsubsectionnumwidth}%
\makeatother

\settocsectionvskip{5pt}
\settocsubsectionvskip{0pt}
\settocsubsubsectionvskip{0pt}
    
\settocsectionformat{\bfseries}
\settocsubsectionformat{\mdseries}
\settocsubsubsectionformat{\mdseries}
\setsectionpnumformat{\bfseries}
\setsubsectionpnumformat{\mdseries}
\setsubsubsectionpnumformat{\mdseries}

\let\oldtableofcontents\tableofcontents
\renewcommand{\tableofcontents}{%
  \vspace*{-\linespacing}
  \oldtableofcontents}

\usepackage{tikz}
\usepackage{comment}
\usetikzlibrary{shapes.misc, decorations.markings, arrows.meta}

\usepackage{amssymb}

\usepackage{subcaption}

\tikzset{
  mid arrow/.style={
    decoration={
      markings,
      mark=at position #1 with {\arrow[sloped, scale=1.2]{Stealth}} 
    },
    postaction={decorate}
  },
  mid arrow/.default=0.5,
  mycurve/.style={smooth, tension=0.3}
}

\usepackage{tikz-cd, tikz-3dplot}
\usetikzlibrary{patterns, knots, calc, math, arrows.meta, decorations, decorations.markings, shapes.misc}
\tikzset{anchorbase/.style={baseline={([yshift=-0.5ex]current bounding box.center)}}}
\tikzstyle directed=[postaction={decorate,decoration={markings,
    mark=at position #1 with {\arrow{>}}}}]
\tikzset{cross/.style={cross out, draw=black, minimum size=2*(#1-\pgflinewidth), inner sep=0pt, outer sep=0pt},
cross/.default={1pt}}
\tikzset{
    partial ellipse/.style args={#1:#2:#3}{
        insert path={+ (#1:#3) arc (#1:#2:#3)}
    }
}

\usepackage{graphicx} 
\usepackage{xcolor}
\usepackage[bookmarks=true, bookmarksopen=true,%
bookmarksdepth=3,bookmarksopenlevel=2,%
colorlinks=true,%
linkcolor=blue,%
citecolor=blue,%
filecolor=blue,%
menucolor=blue,%
urlcolor=blue]{hyperref}

\usepackage{mathtools}
\allowdisplaybreaks[1]

\input{macros}

\title[Skein-valued mirror curves for toric CY3 strips]{Skein-valued mirror curves for toric CY3 strips}
\author{Mingyuan Hu and Vivek Shende}
\date{}

\begin{document}

\begin{abstract}
    For a smooth semi-projective toric Calabi-Yau 3-fold containing no compact surface, we show  the count of all-genus holomorphic curves with boundary on a single Aganagic-Vafa brane is annihilated by a skein-valued quantization of the mirror curve, and that this determines the count.  We give explicit expressions for  the equation and its solution.
\end{abstract}

\maketitle

\vspace{-5mm}

\thispagestyle{empty}

\section{Introduction}

There are at least three approaches to counting holomorphic curves in a toric Calabi-Yau 3-fold.  
The first is a gluing formula for piecing together global invariants from the topological vertex, which counts  holomorphic curves in  $\mathbb{C}^3$ charts \cite{AKMV, Okounkov-Reshetikhin-Vafa, MNOP, LLLZ, MOOP}.  A second is topological recursion, which is a recursive procedure for constructing various differential forms on the mirror curve, whose  integrals recover the holomorphic curve invariants \cite{Bouchard-Klemm-Marino-Pasquetti, Eynard-Orantin, Fang-Liu-Zong}. 
Finally, in certain cases, quantizing the mirror curve gives an operator equation which is believed to annihilate the count of holomorphic curves with boundary ending on a single rank-1 brane \cite{Aganagic-Vafa, ADKMV, AV2,  AENV, Gukov-Sulkowski, Zhou-quantum-mirror, Banerjee-Hock}.  

In the string theory literature, these different approaches each have their own independent justifications.  By contrast, for the first 20 years of mathematical study of such questions, the proofs that the second two approaches in fact give holomorphic curve counts ultimately proceeded by reducing to the topological vertex.  

Recently, a new technique has been introduced which allows an {\em a priori}  derivation of a `skein quantized' mirror, which {\em a priori} annihilates the full open partition function and {\em a priori} dequantizes to the mirror curve \cite{Ekholm-Shende-unknot}.  
Previous applications include \cite{Ekholm-Shende-unknot,Ekholm-Shende-colored,SS24, Ekholm-Longhi-Nakamura, HSZ}, and, notably, a derivation topological vertex itself \cite{ELS}.  
The basic idea is the following: one studies a one-dimensional moduli space of holomorphic curves with boundary on the brane of interest and one positive puncture ``at infinity''.  The boundary of this moduli space is, on the one hand, zero in homology, and, on the other, can expressed in terms only of (1) curves ``at infinity'' -- which are, in examples, relatively easy to determine -- and -- at least in especially fortunate cases -- (2) the desired holomorphic curve count.  This translates into a recursive formula for the counts, which, when packaged in the skein-valued curve-counting formalism of \cite{Ekholm-Shende-skeins}, gives an operator equation.

Here we study smooth semi-projective toric Calabi-Yau 3-folds containing no surfaces.  We give a skein quantization of the mirror, and solve explicitly the corresponding operator equation, in  Theorem \ref{solving the recursion} below.  We show that this equation is given by a count of curves at infinity -- hence annihilates the skein-valued count of curves ending on a filling -- in Theorem \ref{thm:counting at infinity}.   Finally, we give in Theorem \ref{thm: topological vertex for strips}  an independent argument that our solution to the skein recursion  agrees with the result of a topological vertex calculation. 

\vspace{2mm}
{\bf Acknowledgements.}
We thank Tobias Ekholm, Melissa Liu, Siyang Liu, Adrian Petr,  and Eric Zaslow for helpful discussions.
This work was supported by the Villum Fonden grant Villum Investigator 37814.




\renewcommand{\contentsname}{}
\tableofcontents

\section{A skein identity}


\subsection{Skein of the solid torus} 

Here we briefly recall some facts about the HOMFLYPT skein, and in particular about the skein of the solid torus.

For an oriented 3-manifold $M$, its HOMFLYPT skein $\Sk(M)$ is the $\mathbb{Z}[a^\pm, z^\pm]$-module generated by framed links in $M$, modulo the following relations: 
\begin{align}
\vcenter{\hbox{
\begin{tikzpicture}[scale=0.7]
\draw[dotted] (0,0) circle (1);
\draw[ultra thick, ->] ({sqrt(2)/2},{-sqrt(2)/2}) -- ({-sqrt(2)/2},{sqrt(2)/2});
\draw[white, line width=2.5mm] ({-sqrt(2)/2},{-sqrt(2)/2}) -- ({sqrt(2)/2},{sqrt(2)/2});
\draw[ultra thick, ->] ({-sqrt(2)/2},{-sqrt(2)/2}) -- ({sqrt(2)/2},{sqrt(2)/2});
\end{tikzpicture}
}}
\;\;-\;\;
\vcenter{\hbox{
\begin{tikzpicture}[scale=0.7]
\draw[dotted] (0,0) circle (1);
\draw[ultra thick, ->] ({-sqrt(2)/2},{-sqrt(2)/2}) -- ({sqrt(2)/2},{sqrt(2)/2});
\draw[white, line width=2.5mm] ({sqrt(2)/2},{-sqrt(2)/2}) -- ({-sqrt(2)/2},{sqrt(2)/2});
\draw[ultra thick, ->] ({sqrt(2)/2},{-sqrt(2)/2}) -- ({-sqrt(2)/2},{sqrt(2)/2});
\end{tikzpicture}
}}
\;\;&=\;\;
z\;
\vcenter{\hbox{
\begin{tikzpicture}[scale=0.7]
\draw[dotted] (0,0) circle (1);
\draw[ultra thick, <-] ({sqrt(2)/2},{sqrt(2)/2}) arc (135:225:1);
\draw[ultra thick, ->] ({-sqrt(2)/2},{-sqrt(2)/2}) arc (-45:45:1);
\end{tikzpicture}
}}
\;, \label{eq:skeinrel1}
\\
a
\vcenter{\hbox{
\begin{tikzpicture}[scale=0.7]
\draw[dotted] (0,0) circle (1);
\end{tikzpicture}
}}
\;-\;
a^{-1}
\vcenter{\hbox{
\begin{tikzpicture}[scale=0.7]
\draw[dotted] (0,0) circle (1);
\end{tikzpicture}
}}
\;\;&=\;\;
z\;\vcenter{\hbox{
\begin{tikzpicture}[scale=0.7]
\draw[dotted] (0,0) circle (1);
\draw[ultra thick, ->] (0.5,0) arc (0:370:0.5);
\end{tikzpicture}
}}
\;, \label{eq:skeinrel2}
\\
\vcenter{\hbox{
\begin{tikzpicture}[scale=0.25]
\draw[dotted] (0, 0) circle (3);
\draw [ultra thick] (1,-1) to [out=180,in=-90] (0,0);
\draw [ultra thick, ->] (0,0) -- (0,3);
\draw [ultra thick] (1,1) to [out=0,in=90] (2,0) to [out=-90,in=0] (1,-1);
\draw [white, line width=2.5mm] (0,-3) to [out=90,in=-90] (0,0) to [out=90,in=180] (1,1);
\draw [ultra thick] (0,-3) to [out=90,in=-90] (0,0) to [out=90,in=180] (1,1);
\end{tikzpicture}
}}
\;\;&=\;\;
a\;
\vcenter{\hbox{
\begin{tikzpicture}[scale=0.25]
\draw[dotted] (0, 0) circle (3);
\draw[ultra thick, <-] (0, 3) -- (0, -3);
\end{tikzpicture}
}}
\;. \label{eq:skeinrel3}
\end{align}
The existence of the HOMFLYPT invariant of knots is equivalent to the fact that the map sending $1$ to the empty link is an isomorphism: $\mathbb{Z}[a^{\pm}, z^{\pm}] \xrightarrow{\sim} \Sk(S^3)$.

For a surface $S$, we abbreviate $\Sk(S) := \Sk(S \times \mathbb{R})$.  Concatening in the $\mathbb{R}$ factor gives $\Sk(S)$ an algebra structure.  If $M$ is a manifold with boundary, there is similarly an action of $\Sk(\partial M)$ on $\Sk(M)$.  

In this article we will be concerned with the skeins of solid tori $\mathbf{T}$ and their boundaries $\partial \mathbf{T}$.  We will always fix (later determined by the ambient geometry) a choice of orientation of the longitude $\ell$ of the solid torus.  We choose the oriented meridian $m \in \partial \mathbf{T}$ such that the linking number of $m, \ell$ is $+1$.  We will always fix a lift of $\ell$ to $\partial \mathbf{T}$ (this corresponds to the `framing' in discussions of the topological vertex); the space of such lifts is a $\mathbb{Z}$-torsor.  These data determine an embedding $\mathbf{T} \to S^3$, hence a corresponding map 
$\langle\, \cdot\, \rangle_{S^3}: \Sk(\mathbf{T}) \to \Sk(S^3) = \mathbb{Z}[a^{\pm}, z^{\pm}]$.  
These choices induce an identification of the solid torus with the product of an annulus and an interval, so that the longitude is embedded in the annulus.  
This identification gives an algebra structure to $\Sk(\mathbf{T})$; said algebra is commutative,  as can be seen by `rolling one solid torus around the other'.

The meridian $m$ acts via  $\Sk(\partial \mathbf{T}) \circ \Sk(\mathbf{T}) \to \Sk(\mathbf{T})$.  After  setting $z = q^{1/2} - q^{-1/2}$, the action of $m$ 
becomes diagonalizable with distinct eigenvalues $(\lambda,\mu)$
indexed by pairs of partitions \cite{Hadji-Morton}:
$$
\frac{a-a^{-1}}{q^{1/2}-q^{-1/2}}+
(q^{1/2}-q^{-1/2})\left(a C_\lambda(q)-a^{-1}C_\mu(q^{-1})\right),
$$ 
Here, $C_\lambda$ and $C_\mu$ are the `content polynomials' of the partitions.  
(Some  arguments require inverting these eigenvalues, which we do if necessary without comment.) 
We denote the corresponding eigenvectors by $W_{\lambda, \mu}$; these are normalized by fixing the value of $\langle W_{\lambda, \mu} \rangle_{S^3}$ by an appropriate quantum dimension formula. 
The $W_{\lambda} := W_{\lambda, \emptyset}$ span the submodule $\Sk_+(\mathbf{T})$ generated by links everywhere parallel to the positive longitude; for geometric reasons our skeins will always lie $\Sk_+(\mathbf{T})$.  We will also write $\widehat{\Sk}$ for the completion in which we allow infinite sums, so long as there are only finitely many terms with any given positive winding.

Finally, if $\Lambda$ is the algebra of symmetric functions, with scalars extended to $\mathbb{Z}[a^\pm, z^\pm]$, there is an algebra isomorphism $\Lambda \cong \Sk_+(\mathbf{T})$ carrying the Schur function $s_\lambda$ to $W_\lambda$ \cite{Aiston-Morton}.

The algebra structure of $\Sk(\partial \mathbf{T})$ and its action on $\Sk(\mathbf{T})$ are determined explicitly in \cite{Morton-Samuelson}; where it is shown to be the specialization of the elliptic Hall algebra on its polynomial representation.  Particularly useful operators are given by the unique-up-to-isotopy embedded curve $P_{a,b} \subset \partial \mathbf{T}$ in homology class $am + b \ell$; in particular, $P_{1,0}$ and $P_{0,1}$ are the meridian and longitude.  We use the same notation for the corresponding element of $\Sk(\partial \mathbf{T})$.

For manifolds with boundary, we will also sometimes fix oriented marked points on the boundary, and consider the skeins in which we allow tangles entering/exiting at the marked points. 

We will later want to consider the $a=1$ specialization of $\Sk_+(\mathbf{T})$. 
As the eigenvalues of the meridian operator among the $W_{\lambda}$ remain distinct, these remain linearly independent.

\subsection{Skein dilogarithm}
\label{ssec: skein dilog} 

We recall the skein dilogarithm and related identities,
which were introduced in the present context in \cite{Ekholm-Shende-unknot} as the full multiple cover formula for the disk, and the skein recursion characterizing it.  The $U(1)$ specialization recovers the q-dilogarithm, and some of the q-cluster algebra of \cite{FG2} lifts to the present context, and makes contact with various geometric wall crossing phenomena; see \cite{SS23, SS24, HSZ, Nak24, Ekholm-Longhi-Park-Shende}.

\begin{definition}\label{def : skein dilog}
The \emph{exponentiated skein dilogarithm} $\Psi[\xi] \in \mathbb{Q}[\xi]\otimes \widehat{\Sk}_{+}(\mathbf{T})$ is the unique solution to
\begin{equation}\label{eq:dilog-recurrence}
(\bigcirc - P_{1,0} - a\xi P_{0,1}) \Psi[\xi] = 0 
\end{equation} 
of the form $\Psi[\xi] = 1 + \cdots$.
\end{definition}
Explicitly, 
\begin{equation}\label{eq:explicit-skein-dilog}
\Psi[\xi]:= \sum_{\lambda} \prod_{\square \in \lambda}\frac{-q^{-c(\square)/2}\xi}{q^{h(\square)/2} - q^{-h(\square)/2}} W_{\lambda} =  \mathrm{exp}\left( -\sum_d \frac{1}{d} \frac{\xi^d\, P_d}{q^{d/2}-q^{-d/2}} \right).
\end{equation}
For the existence and uniqueness of the solution, and the first expression for it, see \cite{Ekholm-Shende-unknot}.  
In the second expression, the $P_d$ are the images of the power sum symmetric functions under the identification of symmetric functions and $\Sk_+(\mathbf{T})$. 
The equality of the two formulas can be derived by manipulation of symmetric functions, see e.g. \cite[Section 7]{HSZ} and \cite[Theorem 1.1 (d)]{Nak24}. 
We will simply write $\Psi$ for $\Psi[1]$. 

The skein dilogarithm satisfies a relative version of the 3-term recurrence relation \eqref{eq:dilog-recurrence} in a thickened annulus with 2 marked points on the boundary, derived geometrically in \cite{Ekholm-Shende-unknot}, and by direct algebraic manipulation in \cite[Lemma 5.5]{HSZ} or  \cite[Equation (104)]{Nak24}: 
\begin{equation}\label{eq:dilog-relative-recurrence}
\vcenter{\hbox{
\begin{tikzpicture}
\draw[very thick, blue, <-] (0, 0) [partial ellipse = 0 : 360 : 0.6];
\draw[white, line width=5] (0, 0.3) -- (0, 1);
\draw[very thick, ->] (0, 0.3) -- (0, 1);
\draw[very thick] (0, 0) circle (1);
\draw[very thick] (0, 0) circle (0.3);
\node[blue, right] at (1.0, 0){$\Psi$};
\end{tikzpicture}
}}
\;\;=\;\;
\vcenter{\hbox{
\begin{tikzpicture}
\draw[very thick, ->] (0, 0.3) -- (0, 1);
\draw[white, line width=5] (0, 0) [partial ellipse = 0 : 360 : 0.6];
\draw[very thick, blue, <-] (0, 0) [partial ellipse = 0 : 360 : 0.6];
\draw[very thick] (0, 0) circle (1);
\draw[very thick] (0, 0) circle (0.3);
\node[blue, right] at (1.0, 0){$\Psi$};
\end{tikzpicture}
}}
\;\;+\;\;
\vcenter{\hbox{
\begin{tikzpicture}
\draw[very thick, blue, <-] (0, 0) [partial ellipse = 0 : 360 : 0.5];
\draw[white, line width=5] (0, 0.3) to[out=90, in=150] ({0.7*1/2}, {0.7*sqrt(3)/2});
\draw[very thick] (0, 0.3) to[out=90, in=150] ({0.7*1/2}, {0.7*sqrt(3)/2});
\draw[very thick] (0, 0) [partial ellipse = 60 : -240 : 0.7];
\draw[very thick, ->] ({-0.7*1/2}, {0.7*sqrt(3)/2}) to[out=30, in=-90] (0, 1); 
\draw[very thick] (0, 0) circle (1);
\draw[very thick] (0, 0) circle (0.3);
\node[blue, right] at (1.0, 0){$\Psi$};
\end{tikzpicture}
}}
\;.
\end{equation}

The inverse of the skein dilogarithm is given by:
\begin{equation}\label{eq:skein-dilog-inverse}
\Psi[\xi]^{-1} = \sum_{\lambda} \prod_{\square \in \lambda}\frac{q^{c(\square)/2} \xi}{q^{h(\square)/2} - q^{-h(\square)/2}} W_{\lambda}
\end{equation}
which satisfies an analogous 3-term recurrence relation
\begin{equation}
(\bigcirc - P_{-1,0} - a^{-1} \xi P_{0,1}) \Psi^{-1} = 0 
\end{equation}
where $P_{-1,0}$ is the meridian with opposite orientation, and the corresponding relative version.

Given any framed knot $K$ in any oriented manifold $M$, we write $\Psi[\xi](K)$ for the result of inserting $\Psi[\xi]$ in place of the tubular neighborhood of $K$; similarly $\Psi[\xi]^{-1}(K)$.

\subsection{Mutation and an operator equation}

Consider the
torus illustrated in Figure \ref{fig:skeins on torus}.   

We denote by \(Y_l\) and \(X_k\) the indicated paths, winding $l$ or $k$ times around the vertical direction after composition with the indicated path.  We use the same notation for the corresponding elements of the skein module  \(\Sk(\partial \mathbf{T}), c_\pm)\).

\begin{figure}
    \centering
   \begin{tikzpicture}[every path/.style={line width=1.1pt}, scale = .7]
 \draw[black] (0,0) rectangle (6,6);
  \draw[Red, mid arrow = .7] 
    (3,3) .. controls (3,4) and (2.8,5) .. (2.7,6);
  \draw[Red, mid arrow=0.75] 
    (2.7,0) -- (2.1,6);
\draw[Red, mid arrow=. 75] 
    (2.1,0) -- (1.5,6);
\draw[Red, mid arrow=.75] 
    (.7,0) .. controls (.5,2) and (0, 5) .. (0, 6);
\node at (1,4) {\large \color{Red} \(\cdots\)};
\node[above] at (2,6.1) {\color{Red} \(Y_l\)};

\begin{scope}[shift={(6,0)}, xscale=-1]
     \draw[Green, mid arrow = .7] 
    (3,3) .. controls (3,4) and (2.8,5) .. (2.7,6);
  \draw[Green, mid arrow=0.75] 
    (2.7,0) -- (2.1,6);
  \draw[Green, mid arrow=. 75] 
    (2.1,0) -- (1.5,6);
  \draw[Green, mid arrow=.75] 
    (.7,0) .. controls (.5,2) and (0, 5) .. (0, 6);
  \node at (1,4) {\large \color{Green} \(\cdots\)};
  \node[above] at (2,6.1) {\color{Green} \(X_k\)};
\end{scope}

  \draw[YellowOrange, dashed, line width = 1.5pt, mid arrow=0.6] 
    (0, 0) .. controls (0,2) and (3,1) .. (3,3);

  \node[anchor = north west] at (1, 1.3) {\color{YellowOrange} \large \(p\) };

  \node[draw=blue, circle, fill = blue, inner sep=1.5pt] at (0, 0) {};
  \node[draw=blue,  circle, fill = blue, inner sep=1.5pt] at (3, 3) {};

 \node[left] at (0, 0) {\color{blue} \large \(c_-\)};
  \node[right] at (3, 3) {\color{blue} \large\(c_+\)};
\end{tikzpicture}
    \caption{The skein elements $X_k$  and $Y_l$ in \(\Sk(T^2, c_\pm) \).}
     \label{fig:skeins on torus}
\end{figure}
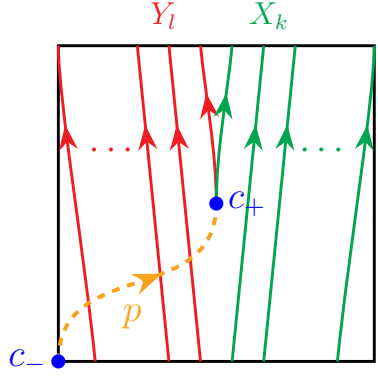

    We also consider the elements 
    \(\overrightarrow{P}_{0, 1}\) and 
    \(\overleftarrow{P}_{0, 1} \in \Sk(T^2, c_\pm)\), defined in Figure \ref{fig:p and p' in relative skein module}.  

    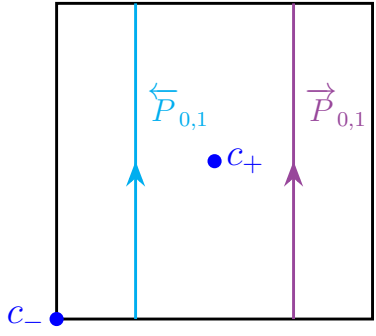
\begin{figure}
        \centering
        \begin{tikzpicture}[every path/.style={line width=1.1pt}, xscale = .7, yscale = .7]
            \draw[black] (0,0) rectangle (6,6);
            
        \node[draw=blue, circle, fill = blue, inner sep=1.5pt] at (3,3) {};
        \node[draw=blue, circle, fill = blue, inner sep=1.5pt] at (0, 0) {};
        \draw[Cyan, mid arrow] (1.5, 0) -- (1.5,6);
        \draw[Purple, mid arrow] (4.5, 0) -- (4.5,6); 
        \node[right] at (4.5, 4) {\textcolor{Purple}{\(\overrightarrow{P}_{0, 1}\)}};
        \node[right] at (1.5, 4) {\textcolor{Cyan}{\(\overleftarrow{P}_{0, 1}\) }}; 
        \node[left] at (0, 0) {\color{blue} \large \(c_-\)};
        \node[right] at (3, 3) {\color{blue} \large\(c_+\)};
        \end{tikzpicture}
        \caption{The skein elements \(\overleftarrow{P}_{0, 1} \) and \(\overrightarrow{P}_{0, 1}\) in \(\Sk(T^2, c_\pm). \) }
        \label{fig:p and p' in relative skein module}
    \end{figure}

\begin{lemma} \label{mutation lemma}
     \(\overrightarrow{P}_{0,1}\) commutes with all \(Y_l\), and \(\overleftarrow{P}_{0, 1}\) commutes with all \(X_k\).  In addition:  
    \begin{align}
        \Ad_{\Psi[t]^{-1}(\overrightarrow{P}_{0,1})} X_k & = X_k - t X_{k+1} \\
        \Ad_{\Psi[t](\overleftarrow{P}_{0, 1})} Y_l &= Y_l - t Y_{l+1} 
    \end{align}     
\end{lemma}
\begin{proof}
    The first assertion is obvious.  The second is an application of Equation \ref{eq:dilog-relative-recurrence}. 
\end{proof}

\begin{theorem} \label{solving the recursion}
Let $\alpha_i, \beta_j$ be any scalars.  Define $A_i, B_j$ by the formulas: 
\begin{align*}
    \prod_i (1 - \alpha_i x)  &= \sum_i A_i x^i\\
    \prod_j (1 - \beta_j x)  &= \sum_j B_j x^j.
\end{align*}


    Then the equation
\begin{equation}\label{eq: recursion relation}
      \left(  \sum_i A_i X_i - \sum_j B_j Y_j \right) \cdot Z = 0 
    \end{equation}
    has a unique solution in $\widehat{\Sk}(\mathbf{T})$ up to scalar multiple, and it is:
    \begin{equation}
        Z = \frac{\prod \Psi[\alpha_i]}{\prod \Psi[\beta_j]}.
    \end{equation}
    Moreover, the equation also has a unique-up-to-scalar solution in $\widehat{\Sk}_+(\mathbf{T})|_{a=1}$, given by $Z|_{a=1}$.
\end{theorem}
\begin{proof}
    The equation
    \begin{equation}
        (X_0 -  Y_0) \cdot Z_0 = 0 \label{eq:x-y=0}.
    \end{equation} 
    can readily be seen to have unique (up to scalar) solution $Z_0 = 1$.   (E.g. close up with the path $p$, then this follows from the fact that the action of the meridian diagonalizes, with only scalars having eigenvalue $1$.) 

    We set 
    $$\widetilde{Z} := \prod \Psi[\alpha_i](\overrightarrow{P}_{0,1}) \prod \Psi[\beta_i]^{-1}(\overleftarrow{P}_{0,1}) \in \widehat{\Sk}(\partial \mathbf{T}).$$ 
    The image of this element in $\widehat{\Sk}(\mathbf{T})$ is $Z$. 
    
    We calculate using Lemma \ref{mutation lemma}
    $$\Ad_{\widetilde{Z}}(X_0 - Y_0) = \sum_i A_i X_i - \sum_j B_j Y_j$$

    The theorem now follows by conjugating Equation \ref{eq:x-y=0} by $\widetilde{Z}$. 
    %
    %
%
%
\end{proof}

\begin{remark}
    The previous occurences of skein-valued cluster transformations have had geometric explanations in terms of disk surgery \cite{SS24} or wall crossing \cite{Ekholm-Longhi-Park-Shende}.  We do not presently have a similar geometric interpretation of the proof of Theorem \ref{solving the recursion}.
\end{remark}

\begin{remark}
Capping with the path \(p\) sends: 
\begin{align*}
    \cap \ p: \Sk(T^2, c_\pm) &\longrightarrow \Sk(T^2)\\
    X_k &\longmapsto P_{1, k}.
\end{align*}
but introduces an asymmetry: $Y_l \not \mapsto  P_{-1, l}$. 
\end{remark}

\begin{remark}
    The HOMFLYPT skein has a specialization to the $U(1)$ skein, given by setting $a=q^{1/2}$, and imposing the additional  relation: 
        $$
    \begin{tikzpicture}[scale = .5]
        \draw[thick, -Stealth] (-1, -1) -- (1, 1);
        \draw[thick, -Stealth] (-.2, .2) -- (-1, 1);
        \draw[thick] (1, -1) -- (.2, -.2);
        \node at (3, 0) {\(= \ q^{1/2}\)} ;

        \begin{scope}[shift = {(5.5, 0)}]
        \draw[thick, -Stealth] (-1, -1) .. controls (-.1, 0) .. (-1, 1);
        \draw[thick, -Stealth] (1, -1) .. controls (.1, 0) .. (1, 1);
        \end{scope}
    \end{tikzpicture}
    $$
    By \cite{przytycki1998q}, the $U(1)$ skein of a surface \(S\) is isomorphic to the quantum torus modeled on the lattice \(H_1(S , \bZ)\). 
   For the surface $\partial \mathbf{T}$, 
   the quantum torus is: 
   \begin{equation}
        \bC[q^{\pm 1/2}] \langle \hat{x}^\pm, \hat{y}^\pm \rangle / {q\hat{x} \hat{y} -  \hat{y} \hat{x}  }
   \end{equation}
    and the specialization is given by  
    \begin{align*}
        P_{a, b} & \longmapsto q^{ab/2} \; \hat{y}^a \, \hat{x}^b . 
    \end{align*}
    Under this map, the operator \( \left(  \sum_i A_i X_i - \sum_j B_j Y_j \right) \) (after capping with the path $p$) goes to
    \begin{equation} \label{quantum mirror curve to strip}
        \hat{y}\prod_i (1- \alpha_i q^{1/2} \hat{x}) + \prod_j (1 - \beta_j q^{1/2}\hat{x})
    \end{equation}
    At \(q = 1\) we obtain a polynomial in two variables:
\begin{equation}\label{eq: mirror curve}
 y\prod_i (1- \alpha_i x) + \prod_j (1 - \beta_j x)
\end{equation}
\end{remark}

%



\section{Recollections on toric geometry}\label{subsec:Toric CY and Strips}

\subsection{Semi-projective toric varieties.} 
We follow the conventions of \cite{Fang-Liu-open}:
\begin{itemize}
    \item Let \(\triangle\) be a fan in a lattice \(N \simeq \bZ^n\), and let \(\triangle(d)\) denote the set of all \(d\)-dimensional cones. 
    \item Denote by \(|\triangle|\) the \emph{support} of \(\triangle\), which is the union of all cones, living in \(N_{\bR} = N \otimes_\bZ \bR \).
    \item Write \(\triangle(1) = \{\rho_1, \rho_2, \dots, \rho_m\}\), where each ray satisfies \(\rho_i \cap N = \bZ_{\ge 0} b_i\). 
    \item Let \(M = \Hom(N, \bZ)\) be the dual lattice. 
\end{itemize}

 We assume that \(|\triangle|\) is convex. By \cite[Proposition 7.2.9]{CLS}, this is equivalent to the corresponding toric variety being semi-projective.

Let \(\widetilde{N} = \bigoplus_{i=1}^m \bZ \widetilde{b}_i\) be a lattice of rank \(r\), and consider the short exact sequence
\[
0 \longrightarrow \bL \longrightarrow \widetilde{N} \xlongrightarrow{\phi} N \longrightarrow 0,
\]
where \(\phi\) sends \(\widetilde{b}_i\) to \(b_i\), and \(\bL = \ker(\phi)\) is a lattice of rank \(k = m - n\).  
Tensoring with \(\bC^*\) yields an exact sequence of algebraic tori
\[
1 \longrightarrow G \longrightarrow \widetilde{T} \longrightarrow T \longrightarrow 1.
\]
Let \(G_{\bR} \simeq (S^1)^{k}\) be the maximal compact torus in \(G\), which acts on \(\bC^m\).  
Denote by
\[
\widetilde{\mu}: \bC^m \longrightarrow \fg_\bR \simeq \bR^{k}
\]
the corresponding moment map, where \(\fg_\bR\) is the Lie algebra of \(G_\bR\).  
Then the toric variety \(X\) can be realized as the symplectic quotient
\begin{equation} \label{eq:symplectic_reduction}
X = \widetilde{\mu}^{-1}(r)\,/\,G_\bR,
\end{equation}
for some point \(r = (r_1, r_2, \dots, r_k)\) in \(\bR^k\).  
Note that \(\fg_\bR\) can be identified with \(A_{n-1}(X)\otimes \bR \simeq H^{1,1}(X, \bR)\).  
The choice of \(r\) thus determines a K\"ahler class of \(X\).

Let \( \{l^{(a)} \}_{a = 1}^{k} \) be a basis for \(\bL  \subset \widetilde{N}\). The action of \(G_\bR\) can be written as
\begin{equation*}
    (t_1, \dots, t_k) \cdot ( x_1, \dots, x_m) = \left( \prod_{a =1 }^k t_a^{l^{(a)}_1} x_1, \dots,  \prod_{a=1}^k t_a^{l_k^{(a)}} x_k \right)
\end{equation*}
and we have:
\begin{equation*}
    \widetilde{\mu} (x_1, x_2, \dots, x_{m}) = \left(\sum_{i  =1}^{m} l_i^{(1)}|x_i|^2, \dots, \sum_{i = 1}^{m} l_i^{(k)} |x_i|^2 \right).
\end{equation*}

\subsection{Toric CY3}

Now assume that \(X\) is a toric Calabi--Yau threefold. 
Then \(m = k + 3\), and the Calabi--Yau condition is equivalent to
\[
\sum_i l^{(a)}_i = 0,
\]
for all \(a = 1, \dots, k\). 
In other words, all the vectors \(b_i\) lie in the plane \(\{z = 1\}\).

The Calabi--Yau condition gives rise to a two-dimensional subtorus 
\(T' \subset T\), acting trivially on the canonical bundle \(K_X\). 
Consider the moment map of the real torus \(T'_\bR \subset T'\):
\begin{equation}\label{eq:moment map to R2}
\mu': X \longrightarrow \bR^2.
\end{equation}
The set of critical values of \(\mu'\) forms a (balanced) trivalent graph, 
called the \emph{formal toric Calabi--Yau (FTCY)}. 
Assuming all victors \(b_i\) live on the \( \{z = 1\} \) plane, the intersections of the cones in \(\triangle\) with this plane give rise to a \(2\)-dimensional polygon with a triangulation (see Figure~\ref{subfig:toric fan} for example).  The dual graph of the triangulation is the FTCY graph.

The FTCY graph can also be identified with the tropicalization of the mirror curve.

\subsection{Aganagic--Vafa branes}
In \cite{Aganagic-Vafa}, Aganagic and Vafa introduced a class of Lagrangian submanifolds of \(X\), given by the equations 
\begin{equation} \label{general AV brane}
    \sum_{i=1}^{k+3} \hat{l}_i^1 |x_i|^2 = c_1, \quad \sum_{i=1}^{k+3} \hat{l}_i^2 |x_2|^2 =c_2, \quad \sum_{i=1}^{k+3} \phi_i = \text{const},
\end{equation}
where \(\phi_i = \arg(x_i) \), \(\hat{l}_i \in \bZ\), and 
\[\sum_{i=1}^{k+3} \hat{l}^\alpha_i = 0, \quad \alpha = 1, 2.\] 
The real numbers \(c_i\) are chosen so that the image of this Lagrangian under the map \eqref{eq:moment map to R2} is a point lying on the FTCY graph of \(X\). 
Topologically, such a Lagrangian is homeomorphic to \(S^1 \times \bR^2\).

We will take the Aganagic--Vafa brane sitting on an external edge, as illustrated in \textcolor{Blue}{blue} in Figure~\ref{fig:toric fan of a strip}. 
It is given in coordinates by the following specialization of Equation \ref{general AV brane}: 
\begin{equation} \label{which AV brane}
    |x_3|^2 - |x_1|^2 = c, \quad |x_2|^2 - |x_1|^2 = 0.
\end{equation}
where \(c\) is a positive real number. 
Denote this Lagrangian by \(L_{AV}\).

Here the direction of the blue ray indicates the ``framing'', which 
will be discussed in Section~\ref{subsec:top_vertex}. 

\begin{example}The resolved conifold is given by the FTCY:

\begin{center}
\begin{tikzpicture}[line width = 1pt]
    \draw (-1, 0) -- (0, 0) -- (0, 1);
    \draw (0, 0) -- (1, -1) -- (2, -1); 
    \draw (1, -1) -- (1, -2); 
\end{tikzpicture}
\end{center}
It can be written as: 
    \begin{equation*}
        X = \{ |x_1|^2 - |x_2|^2 - |x_3|^2 + |x_4|^2 = r\} / U(1)
    \end{equation*}
    where \(U(1)\) acts by
    \[t: (x_1, x_2, x_3, x_4) \longrightarrow (t x_2, t^{-1} x_2, t^{-1} x_3, t x_4). \]
Under the ``conifold transition'', the Aganagic--Vafa brane \(L_{AV}\) becomes the conormal to the unknot in the three sphere.
\end{example}

\subsection{Strips}

We will be interested in semi-projective  toric Calabi--Yau threefolds which contain no compact surfaces, or equivalently, for which the FTCY is a tree. 
Following \cite{IKP,PS19}, we refer to these as \emph{strips}.

One observes that the \(X\) is a strip if and only if the dual triangulation of the FTCY contains no interior integral points. Thus, by the convexity condition, we can always assume that a strip has the form in Figure~\ref{fig:strip} (up to an \(GL(2, \bZ)\) transformation), i.e. ,
all vectors \(b_i\) lie in a ``strip'' \( \{0, 1\} \times \bZ \times \{1\}\). The corresponding FTCY is drawn in \textcolor{Purple}{purple}.

\begin{figure}
    \centering
    \begin{tikzpicture}[ scale = 2]
    \begin{scope}[line width = 1pt]
        \draw (0, 0) -- (0, -1) -- (4, -1)-- (5, 0)  -- (0, 0);
        \draw (1, 0) -- (0, -1);
        \draw (1, 0) -- (1, -1);
        \draw (1, 0) -- (2, -1);
        \draw (1, 0) -- (3, -1);
        \draw (2, 0) -- (3, -1);
    
        \node at (3.5, -.5) {\large \(\cdots\)};

    \end{scope}
    
         \draw[->] (-0.5,0) -- (6,0) node[right] {$y$};
  
        \draw[->] (0,0.5) -- (0,-2) node[below] {$x$};
        
        \node[anchor = south] at (5, 0) {\((0, l)\)};
\fill (0,0) circle (1pt) node[above left] {$b_1$};
\fill (1,0) circle (1pt) node[above] {$b_3$};
\fill (2,0) circle (1pt) node[below] {$b_7$};
\fill (5,0) circle (1pt) node[below right] {$b_{k+3}$};

\fill (0,-1) circle (1pt) node[below left] {$b_{2}$};
\fill (1,-1) circle (1pt) node[below] {$b_{4}$};
\fill (2,-1) circle (1pt) node[below] {$b_{5}$};
\fill (3,-1) circle (1pt) node[below] {$b_{6}$};

\begin{scope}[Purple, line width=1pt]
    \draw (-.8, -.5) -- (.3, -.5) -- (.5, -.7) -- (1.3, -.7) -- (1.6, -.4) -- (1.7, -.2) -- (2.2, .3);
    \draw (.3, -.5) -- (.3, .5);
    \draw (.5, -.7) -- (.5, -1.5); 
    \draw  (1.3, -.7) -- (1.3, -1.5); 
    \draw (1.6, -.4) -- (1.6, -1.5); 
    \draw (1.7, -.2) -- (1.7, .5);
\end{scope}

\draw[line width = 1.5 pt, Blue] (-.6, -.5) -- +(0, .8);

    \end{tikzpicture}

    \caption{A example of a strip. The \textcolor{Purple}{purple} graph is the corresponding FTCY, and the Aganagic--Vafa brane \(L_{AV}\) is illustrated in \textcolor{Blue}{blue}.}
    \label{fig:toric fan of a strip}
\end{figure}
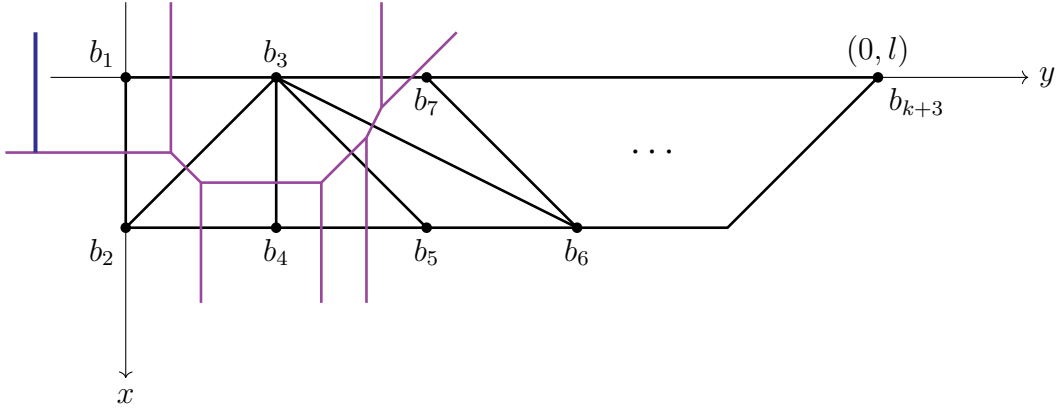

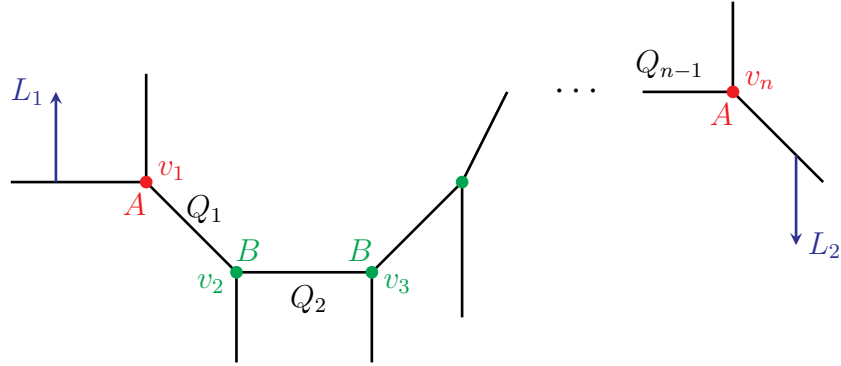
\begin{figure}
    \centering
 
\begin{tikzpicture}[every path/.style={line width=1pt}, xscale = 1.2, yscale = 1.2]

    \begin{scope}[shift = {(-1, -1)}]
    \draw (-1.5, 0) -- (0, 0) -- (0, 1.2);
    \draw (0,0) -- (1, -1); 
    \draw (1, -1) -- (1, -1); 
    \draw (1, -1) -- (1, -2);
    \draw (1, -1) -- (2.5, -1);
    \draw (2.5, -1) -- (2.5, -2);
    \draw (2.5, -1) -- (3.5, 0);

    \draw (3.5, 0) -- (3.5, -1.5);
    \draw (3.5, 0) -- (4, 1); 
    \fill[Green] (3.5, 0) circle (2pt);

    \node (A)  at (0, 0) {};
    \fill[Red] (A) circle (2pt);
    \node [anchor = north] at (A.west) {\color{Red} \(A\) };
    \node [anchor = west] at (A.north) {\color{Red} \(v_1\)};
    \node (B) at (1, -1) {};
    \fill[Green] (B) circle (2pt);
    \node [anchor = south] at (B.east) { \color{Green} \(B\)}; 
    \node [anchor = east] at (B.south) {\color{Green} \(v_2\)};
    \node (B1) at (2.5, -1) {};
    \fill[Green] (B1) circle (2pt);
    \node [anchor = south] at (B1.west) {\color{Green} \(B\)}; 
    \node [anchor = west] at (B1.south) {\color{Green} \(v_3\)};

    \draw[Blue, -stealth] (-1, 0) -- (-1, 1);
    
    \node[left] at (-1, 1) {\textcolor{Blue}{\(L_1\)}};

    \node at (.65, -.3) {$Q_1$};
    \node at (1.8, -1.3) {\(Q_2\)};
    
    \end{scope}
    
    \node[right] at (6.2, -1.7) {\textcolor{Blue}{\(L_2\)}};
  
    \node at (3.8, 0) {\large $\cdots$} ;

    \draw (4.5, 0) -- (5.5, 0) -- (6.5, -1);
    \draw (5.5, 0) -- (5.5, 1);
    \node (An) at (5.5, 0) {};
    \fill[Red] (An) circle (2pt);
    \node [anchor = north] at (An.west) {\color{Red} \(A\)}; 
    \node [anchor = west] at (An.north) {\color{Red} \(v_n\)};

    \draw[Blue, -stealth] (6.2, -.7) -- (6.2, -1.7);

    \node at (4.8, .3) {\(Q_{n-1}\)}; 
    \end{tikzpicture}

    \caption{A strip. The blue edges are the Aganagic--Vafa branes.}
    \label{fig:strip}
\end{figure}

Now consider the FTCY of a strip, as illustrated in Figure \ref{fig:strip}. Each vertex is adjacent to an edge pointing upward (type \(A\), drawn in red) or an edge pointing downward (type \(B\), drawn in green). 
We order the edges and vertices from left to right. Denote the vertices by \(v_1, v_2,  \cdots, v_n\), and the K\"ahler parameters corresponding to the edges by \(Q_1, Q_2, \dots, Q_{n-1}\). Without loss of generality, we assume \(v_1\) is of type \(A\).  Following \cite{PS19}, we denote
\begin{equation*}
   Q_{i,j} = \begin{dcases}
        \prod_{l=i}^{j-1} Q_l & \text{ \(1\) if \(i \ge j-1\).}\\
        1 &       \text{ \(1\) if \(i \ge j\).}
   \end{dcases}
\end{equation*}
If the \(i\)-th vertex is of type \(A\),  we will write:
\begin{equation}
\alpha_i :=  Q_{1, i} . 
\label{eq:def alpha}
\end{equation}
Similarly, if the \(j\)-th vertex is of type \(B\), take
\begin{equation}\beta_j :=  Q_{1, j} .
\label{eq:def beta}
\end{equation}
(Note that for each $i$, either $\alpha_i$ or $\beta_i$ is not defined.) 

In these variables, the mirror curve for the strip was identified in \cite{PS19} to be 
Equation \eqref{eq: mirror curve}, and the `quantum mirror' was identified in \cite{Banerjee-Hock} to be given by \eqref{quantum mirror curve to strip}.\footnote{More precisely, in those references $x$ and $y$ are switched, because they are working with respect to a different Aganagic-Vafa brane.}

\section{Skein valued mirror curves from geometry}

Fix a CY3 strip $X$, and the Aganagic-Vafa brane $L$ of \eqref{which AV brane}.  
In this section, we determine the skein-valued count of curves at infinity asymptotic to a single index zero Reeb chord of $L$.  The result is an instance of the operator \eqref{eq: recursion relation}.  We show also that $X$ has only positive index Reeb orbits and $L$ has only positive index Reeb chords.  It then follows from stretching considerations 
that said operator annihilates the skein-valued count of compact curves in $X$ ending on $L$, and hence that said count is determined by Theorem \ref{solving the recursion}.

\subsection{Recollections on skein-valued curve counting}

  We review the setup of the skein-valued curve counting of \cite{Ekholm-Shende-skeins}, and also the conditions under which counting curves at infinity determines a skein-valued mirror \cite{Ekholm-Shende-unknot, SS23}.

  Let $X$ be a symplectic 6-manifold with trivialized first Chern class.  Let  $L \subset X$ be a smooth Lagrangian submanifold with trivialized Maslov class.  Fix an almost complex structure $J$, standard in a neighborhood of $L$.  
  We require that $(X, L, J)$  `bounded geometry' for Gromov compactness arguments; typically (and here) this is ensured by demanding $(X, L, J)$ asymptotically convex conical at infinity; in this case we have also the `symplectic field theory' compactness results \cite{BEHWZ}. 

  We fix a spin structure on $L$.  In the general setup of \cite{Ekholm-Shende-skeins}, one fixes some additional data -- a vector field on $L$ and compatible $4$-chain bounding $2L$ -- however, here we will specialize the skein at $a=1$, in which case this data is not necessary (see \cite[Sec. 4.2]{Ekholm-Longhi-Park-Shende}).\footnote{As with \cite{Ekholm-Shende-unknot, Ekholm-Shende-colored, Ekholm-Longhi-Nakamura, ELS}, even if we retained the $a$ variable, we would find nevertheless that, for the right choice of 4-chain, no $a$'s actually appear in the curve counts.  Thus we just omit it from the beginning for simplicity.} 

Fix a class $d \in H_2(X, L)$.  A map $u\colon (C, \partial C) \to (X, L)$ is said to be {\em bare} if $u^{\ast}(\omega)$ gives positive symplectic area to every irreducible component of $C$.  We write $\chi(u)$ for Euler characteristic of a smoothing of the domain of $u$.  Given some system of (weighted, multivalued) perturbations for the holomorphic curve equation, if $u$ is a transverse solution to the perturbed equation, we write $\mathrm{wt}(u)$ for the local contribution of this solution.  
After \cite{Ekholm-Shende-skeins, Ekholm-Shende-ghost, Ekholm-Shende-bare}, that there is a system of perturbations to the holomorphic curve equation so that the following sum over bare solutions is well defined and deformation invariant: 
\begin{equation} \label{skein count definition}
    Z_{X, L, d} = \sum_{[u] = d} \mathrm{wt}(u) \cdot z^{-\chi(u)} \cdot [u(\partial \Sigma)] \in \Sk_{a=1}(L) \otimes \mathbb{Q}[[z]]
\end{equation}

To sum over $d$, one generally must extend scalars by the Novikov ring $\mathbb{Q}[ H_2(X, L)]$.
We may omit this completion from the notation, and write $Z_{X, L} = \sum_d Z_{X, L, d} \in \Sk(L)$.

Finally, we recall from \cite{Ekholm-Shende-unknot} how counting at infinity gives a skein-valued operator equation.  We now ask that $X, L$ are asymptotically conical at infinity (hence has ideal boundary modeled by a stable Hamiltonian structure and Lagrangian therein), 
and that all Reeb orbits of $\partial X$ and Reeb chords of $\partial L$ have strictly positive index.  Consider a chord $c$ of $L$ of index one.  The moduli $M(c)$ of holomorphic curves with boundary on $L$ with one positive puncture at $c$ is (after perturbation) a dimension one manifold, which may be compactified using the usual Gromov compactification (giving boundary breakings in the interior of $X$) together with the SFT compactness results (when energy or genus escapes to $\partial X$).  The sum of all contributions will be zero, and the boundary breakings are already cancelled by the skein-valued counting.  On the other hand, since there are only positive index Reeb chords and orbits, the only SFT breakings are when some components of the curve escapes entirely.   In sum:

\begin{lemma} \cite{Ekholm-Shende-unknot} \label{why recursion}
Suppose $(X, L)$ is asymptotically conic, and all Reeb chords or orbits of $(\partial X, \partial L)$ 
are of positive indices. 
Then 
let $A_{\partial X, \partial L}(c) \in \Sk(\partial L; c_{\pm})$ be the skein-valued count of $\R$-families of curves in $(\partial X \times \R, \partial L \times \R)$ with one positive puncture at $c$, and let $Z_(X, L)$ be the count of compact holomorphic curves in $X$ with boundary on $L$.  Then
\begin{equation}  \label{boundary equation}
A_{\partial X, \partial L}(c) \cdot Z_{X, L} = 0.
\end{equation}
\end{lemma}
\begin{remark}
    In \cite{Ekholm-Shende-unknot}, we assumed $(\partial X, \partial L)$ was contact; here it will be only stable Hamiltonian.  As this is what is required for the key compactness result \cite{BEHWZ}, the proof of Lemma \ref{why recursion}  from \cite{Ekholm-Shende-unknot} carries through without change.
\end{remark}

\begin{remark}
The Lemma should be understood as a trivial special case of a yet-to-be-constructed skein-valued SFT.  The next simplest sort of phenomena occurs when there are some index zero Reeb chords; some examples of this form have been treated \cite{Ekholm-Shende-colored, Ekholm-Longhi-Nakamura, ELS}, but so far only by ad-hoc cancellation of the terms resulting from the additional breakings.
We will explain in Section \ref{sec: why strips} below that toric CY3s other than strips will have index zero Reeb orbits; we do not presently know how to derive recursions in that context.     
\end{remark}

\subsection{Moment maps and symplectic reduction}

Let us follow the notation in Section~\ref{subsec:Toric CY and Strips}. Let \(X\) be a strip, and \(T\simeq (\bC^*)^3 \) be the open dense torus, inside which there is the real torus \(T_\bR\).  Denote the moment map of \(T_\bR\) by 
\begin{equation}\label{eq:moment map}
\mu: X \longrightarrow M_\bR \simeq \bR^3.
\end{equation}

The one dimensional cones of the corresponding fan are generated by \( \{ b_i \}\). We fix a basis of \(N\), such that all \(b_i\) live in the ``strip'' \( \{0, 1\} \times \bZ \times \{1\} \), and assume that \(b_1 = (0, 0, 1)\), \(b_2 = (1, 0, 1)\), and \(b_3 = (0, 1, 1)\).

Recall that \(X\) can be defined as a symplectic reduction~\eqref{eq:symplectic_reduction}. 
One can easily check under coordinates, the moment map~\eqref{eq:moment map} can be written as
\[
\mu: [(x_1, x_2, x_3, \cdots)] \longmapsto \left( |x_3|^2 - |x_1|^2, |x_2|^2 - |x_1|^2, |x_1|^2 \right),
\]
and the first two coordinates give the map \(\mu'\) in \eqref{eq:moment map to R2}. 

The image of \(\mu\) is a cone in \(M = N^\vee\). Now choose a direction \(w \in N\) living in the interior of \(|\triangle|\). Then take 
\[
f_w (x) = \langle w, \mu(x) \rangle
\] 
Then the ideal boundary of \(X\) is 
\[
    \partial_\infty X = f_w^{-1} (R), \qquad R \gg 0.
\]
Topologically \(X\) decomposes as
\[
X =  \{f_w(x) \le R\} \bigcup \partial_\infty X \times [0, \infty) 
\]

From another perspective, the direction \(w \in N\) gives rise to a subtorus \(S^1 \subset T_\bR\), and \(f_w\) is  the corresponding Hamiltonian function. 
The condition that \(w\) lives in the interior of \(|\triangle|\) is equivalent to \(f_w\) being proper. An example is illustrated in Figure~\ref{fig:moment map}. 
\begin{figure}
    \centering
    \begin{subfigure}[t]{.55\linewidth}
        \includegraphics[scale =.5]{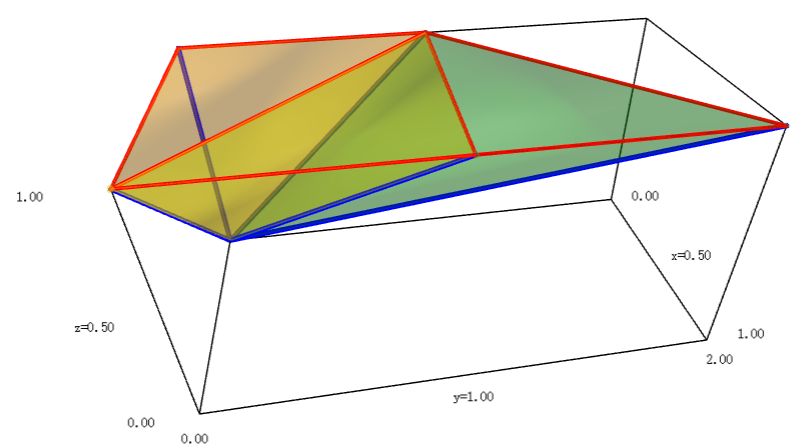}
        \caption{The toric fan}
        \label{subfig:toric fan}
    \end{subfigure}
    \begin{subfigure}[t]{.44\linewidth}
        \centering
        \includegraphics[scale = .5]{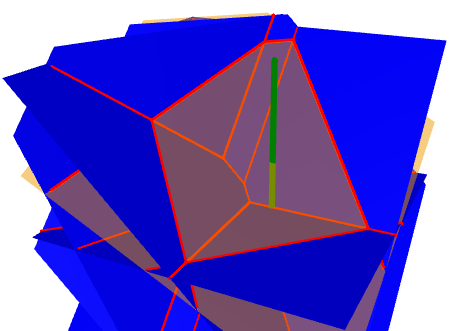}
        \caption{The moment cone. The orange plane is the slice \(\left\{m \in M \mid \langle w , m\rangle = R\right\} \), and the green ray pointing upward is the image of \(L_{AV}\). }
        \label{subfig:moment cone}
    \end{subfigure}
    \caption{An example of the moment cone}
    \label{fig:moment map}
\end{figure}

We choose the direction \(w\) to be 
\[
w = (1, 1, 2) = b_2 + b_3. 
\] 
The Hamiltonian flow of \(f_w\)  is given by:
\[
\theta : [(x_1, \cdots, x_{k+3})] \longmapsto [( x_1, e^{i \theta } x_2, e^{i \theta} x_3, \cdots, x_{k+3})]. 
\]
The quotient of \(\partial_\infty X\) by this \(S^1\)-action is the symplectic reduction presentation of a toric surface (as recalled in Section~\ref{subsec:Toric CY and Strips}):
\[
P = \partial_\infty X/ S^1
\]
whose toric data are shown in Figure~\ref{fig:the toric surface}. The one dimensional cones are generated by:
\begin{align*}
    u_i &= (i-1, i), \qquad i = 0, 1, \dots, k; \\
    u_j' &= (j, j-1), \qquad j= 0, 1, \dots, l. 
\end{align*}
where \(k+1\) and \(l+1\) are the numbers of \(b_i\) lying in \( \{0\} \times \bZ \times \{1\}\) and \( \{1\} \times \bZ \times \{1\}\), respectively.  
The toric surface \(Q\) may have a single orbifold singularity, located in the toric chart corresponding to the cone spanned by \( u_k \) and \( u'_l\).

\begin{figure}
    \centering

    \begin{subfigure}[t]{.5\linewidth}
    \centering
    \begin{tikzpicture}

        \draw[-Stealth, thick] (0, 0) -- (1, 0);
        \draw[-Stealth, thick](0, 0) -- (0, 1);
        \draw[-Stealth, thick] (0, 0) -- (0, -1);
        \draw[-Stealth, thick] (0, 0) -- (-1, 0);
        \draw[-Stealth, thick] (0, 0) -- (2, 1);
        \draw[-Stealth, thick] (0, 0) -- (1, 2);

        \draw[-Stealth, thick] (0, 0) -- (3, 4);
        \node[above left] at (3, 4) { \(u_k\)};
        \draw[-Stealth, thick] (0, 0) -- (4, 3);
        \node[anchor = north west] at (4, 3) {\(u'_l\)};

        \node[above] at (-1, 0) {\(u_0\)};
        \node[above left] at (0, 1) {\(u_1\)};
        \node[above left] at (1, 2) {\(u_2\)};

        \node[right] at (0, -1) {\(u'_0\)};
        \node[below right] at (1, 0) {\(u'_1\)};
        \node[below right] at (2, 1) {\(u'_2\)};
        
        \node at (1.5, 2.7) {\(\cdots\)};
        \node at (2.7, 1.5) {\(\cdots\)};
        
        \draw[step=1cm,gray,very thin] (-1,-1) grid (5,5);

        \node[anchor = north east] at (0,0) {\tiny \( (0, 0)\)};
        \fill[gray] (0,0) circle (2pt);
        
    \end{tikzpicture}
    
    \caption{The toric fan}
    \label{subfig:fan of toric suface}
    \end{subfigure}
    \begin{subfigure}[t]{.45\linewidth}
    \centering
    \begin{tikzpicture}[scale = .9]
        \draw[line width = 1pt] (4.5, 5.25 - 2/3) --(3.5, 5.25) -- (2, 6) -- (0, 6) -- (0, 0) -- (6, 0) -- (6,2) -- (5.25, 3.5); 
        \node at (5, 4) {\large\(\cdots\)};
        \draw[red] (1.5, 1.5) -- (1.5, 6);
        \draw[red] (1.5, 1.5) -- (6, 1.5);
        \draw[red] (1.5, 1.5) -- (3.4,5.3);
        \draw[red] (1.5, 1.5) -- (5.3,3.4);
        \draw[red] (1.5, 1.5) -- (1.5,0);
        \draw[red] (1.5, 1.5) -- (0,1.5);
        \draw[red] (1.5, 1.5) -- ((3.8,5.05);
        \fill[blue] (1.5, 1.5) circle (2pt);

        \node[below left] at (1.5, 1.5) {\color{blue} \(p\)}; 
        \node[above] at (.6, 1.5) {\color{red} \(\delta_0\)};
        \node[right] at (1.5, .6) {\color{red} \(\delta_0'\)};
        \node[below] at (4.5, 1.5) {\color{red} \(\delta_1'\)};
        \node[left] at (1.5, 4.5) {\color{red} \(\delta_1\)};
        \node at (2.5, 4.5) {\color{red} \(\delta_2\)};
        \node at (4.5, 2.5) {\color{red} \(\delta_2'\)};

        \node[below] at (3, 0) {\(D'_0\)};
        \node[left] at (0, 3) {\(D_0\)};
        \node[above] at (.9, 6) {\(D_1\)};
        \node at (3, 5.9) {\(D_2\)};
        \node[right] at (6, .9) {\(D'_1\)};
        \node at (5, 1) at (6, 3) {\(D'_2\)};
        
    \end{tikzpicture}
    
    \caption{The moment polytope}
    \label{subfig:moment polytope of toric suface}
    \end{subfigure}
    \caption{The toric surface \(P = \partial_\infty /S^1\)}
    \label{fig:the toric surface}
\end{figure}
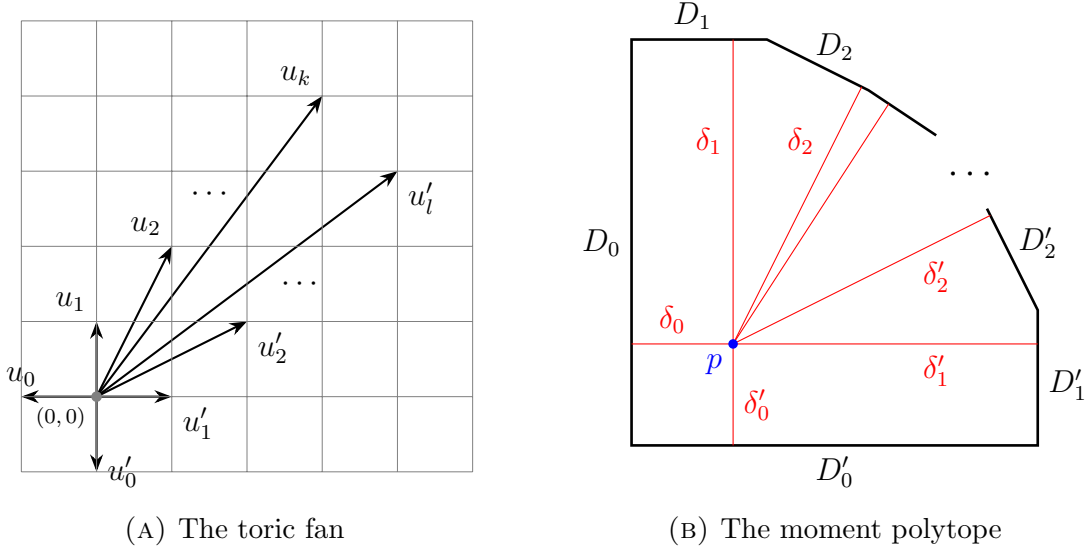

\subsection{Counting at  infinity}

We recall that $\partial_\infty X$, or indeed the total space of any principal circle bundle over a symplectic manifold, is a stable Hamiltonian structure -- where the 1-form $\alpha$ is the  connection \(1\)-form \(\alpha\) of the circle bundle, and the 2-form $\omega$ is the pullback of the symplectic form on the base.  The Reeb dynamics is just the principal $S^1$ action.  

\begin{lemma}\label{lem:c1_pi} 
The circle fibration $
     \pi: \partial_\infty X \longrightarrow P$
     has  $c_1(\pi ) = \frac{1}{2} c_1(TP)$. 
In particular, \(c_1(\pi)\) is primitive in \(H^2(P, \bZ)\). Hence all the Reeb orbits are contractible. 
\end{lemma}

\begin{proof}
    From Thom--Gysin sequence
    \[
    \longrightarrow 
     H^0 (P, \bZ) \stackrel{ \cup\, c_1(\pi)}{\longrightarrow} H^2 ( P , \bZ) 
     \stackrel{\pi^*}{\longrightarrow} H^2 ( \partial_\infty X, \bZ) \longrightarrow  H^1 (P, \bZ) = 0
    \]
    we know that 
    $
    H^2 \left(\partial_\infty X, \bZ\right) = H^2 (P, \bZ)/ c_1(\pi).
    $
    Since the strip \(X\) is Calabi-Yau, we know that \(\pi^* c_1(P) = 0\). Hence \( \pi^* c_1 (P)\) must have the form \( k c_1 (\pi)\) for \(k \in \bZ \). 
    Let \(D_0\) be the divisor corresponding to \(u_0 = (-1, 0)\) as in Figure~\ref{fig:the toric surface}. One can easily compute 
    $
    \langle c_1 (P) , [D_0] \rangle = 2
    $
    while 
    $
    \langle c_1 (\pi) , [D_0] \rangle = 1.
    $
    This completes the proof. 
\end{proof}


\begin{lemma} \label{reeb positivity}
    After generic perturbation, the Reeb orbits of $\partial_\infty X$ have positive indices, as do the Reeb chords of $\partial_\infty L$.  
    In particular, Lemma \ref{why recursion} applies to an Aganagic-Vafa brane in a strip. 
\end{lemma}
\begin{proof}
    It is well known how to compute the Conley-Zehnder indices for a prequantization bundle; see e.g. an exposition in \cite{VK15} and  \cite{Hong19} where the orbifold case is worked out in detail. 
    The same argument will work here.

    Since \(c_1 (\pi)\) is primitive, we may pick a homology class $H$ with 
    $c_1(\pi) \cap H = 1$. 
    Since $P$ is simply connected, \(H_2 (P, \bZ) \cong \pi_2(P)\), so we may represent $H$ by a map from a sphere $i: S \to P$, and to pass through any given point $q$.  
    Now 
    consider the orbit \(\gamma_q\) lying over a point \(q \in P\).
    The map $S \to P$ lifts to a map from a closed disk
    $
    \tilde{i}: D \longrightarrow \partial_{\infty} X
    $; note $\tilde{i}(\partial D)$ is the \(S^1\) orbit over \(q\). By \cite[Section 2.7]{McDuff-Salamon}, we have
    $$
        \mu_{CZ} (\gamma_q)
        = 2\langle \iota^* c_1 (\xi) , S \rangle \\
        = 2\left\langle \iota^*c_1 (  TP) , S \right \rangle \\
        \stackrel{\text{Lemma~\ref{lem:c1_pi}}}{{=\joinrel=\joinrel=\joinrel=\joinrel=}} 4 \langle c_1 (\pi), \iota_*[S] \rangle \\
         = 4
    $$

    Finally, the perturbation by a Morse function decreases the Conley-Zehnder index by at most half the dimension of the base manifold (see for example \cite[Lemma 2.4]{VK15}), in this case $2$.  Thus the indices remain positive after perturbation.
   

    Now we consider the chords.  The Aganagic--Vafa brane \(L_{AV}\) intersects the ideal boundary along a  torus  \(\bT\). The 
    \(S^1\)-quotient induces a two-to-one covering from to a Lagrangian torus  which is a fiber of the moment map: \( \bT \xrightarrow{2:1} \bT_0 \subset P\).  Before  perturbing, the shortest Reeb chords are equally spaced segments of Reeb orbits.  It follows that the perturbed chords must all have positive index if the orbits do. 
\end{proof}


\begin{theorem}\label{thm:counting at infinity}
     In the symplectization $(\R \times \partial_\infty X, \R \times \partial_\infty L)$,
     the skein-valued count of holomorphic curves with one positive puncture at the (perturbed) index one Reeb chord of $L$ is given by Equation    
    \eqref{eq: recursion relation}.
\end{theorem}
\begin{proof}
    The count in question is of the $\R$-invariant families of curves, and so is computed by counting in the symplectic reduction to $P := \partial_\infty X / S^1$.  We write $\bT_0 \subset P$ for the image of $\bT$. 

    In this case, \(P\) is a semi-Fano (orbifold) toric surface. 
    Consider the moment polytope of \(P\),  as in Figure~\ref{subfig:moment polytope of toric suface}. The Lagrangian \(\bT_0\) is the preimage of the blue point \textcolor{blue}{\(p\)}. 
    
    One sees immediately that the only possible embedded rigid curves are the Maslov $2$ disks 
    lying over each \textcolor{red}{red ray}, starting at \(p\) and perpendicular to a boundary divisor.     
    Denote the disk corresponding to the divisor \(D_i\) (resp. \(D'_i\)) by \(\delta_i\)
    (resp. \(\delta'_i\)). 
    One sees directly that the boundary of $\delta_i$ (resp. \(\delta_i'\))  lifts to the curve \(X_i\) (resp. \(Y_i\)) on $\bT$. 
        
    One sees by numerical considerations that the only possible rigid (after perturbation) contributions must be of genus zero, and can only arise by gluing these embedded disks to spheres with self intersection $\le -2$  contained in the boundary divisor (in this case, the only such sphere have self intersection $-2$).  This excludes the divisors $D_0, D_0'$, and also, by \cite[Thm. 5.2]{CP14}, the  boundary components touching the orbifold point. 

    We now consider the remaining allowable curves.  They are not strictly speaking admissible in the skein-valued curve counting formalism (as the curves are not embedded, and may not be transversely cut out in moduli), but as they already have embedded boundary in $\bT_0$, the curves which result after perturbation will have the same boundary,  and the number of such curves can be computed using the Kuranishi structure on the unperturbed moduli spaces.  

    \begin{figure}
    \begin{tikzpicture}[scale =.8]

        \draw[-Stealth, thick](0, 0) -- (0, 1);
        \draw[-Stealth, thick] (0, 0) -- (-1, 0);
        \draw[-Stealth, thick] (0, 0) -- (1, 2);

        \draw[-Stealth, thick] (0, 0) -- (3, 4);
        \node[above left] at (3, 4) { \(u_k\)};

        \node[above] at (-1, 0) {\(u_0\)};
        \node[above left] at (0, 1) {\(u_1\)};
        \node[above left] at (1, 2) {\(u_2\)};

        \node at (1.5, 2.7) {\(\cdots\)};

        \draw[step=1cm,gray,very thin] (-1,-1) grid (4,4);

        \fill[gray] (0,0) circle (2pt);

        \begin{scope}[shift = {(6, 0)}]
               \draw[-Stealth, thick] (0, 0) -- (1, 0);
   
        \draw[-Stealth, thick] (0, 0) -- (0, -1);
        \draw[-Stealth, thick] (0, 0) -- (2, 1);

        \draw[-Stealth, thick] (0, 0) -- (4, 3);
        \node[anchor = north west] at (4, 3) {\(u'_l\)};

        \node[right] at (0, -1) {\(u'_0\)};
        \node[below right] at (1, 0) {\(u'_1\)};
        \node[below right] at (2, 1) {\(u'_2\)};
    
        \node at (2.7, 1.5) {\(\cdots\)};
        
        \draw[step=1cm,gray,very thin] (-1,-1) grid (4,4);

        \fill[gray] (0,0) circle (2pt);
        \end{scope}
    \end{tikzpicture}
    \caption{The fans of two toric surfaces}
    \label{fig:two toric surfaces}
    \end{figure}
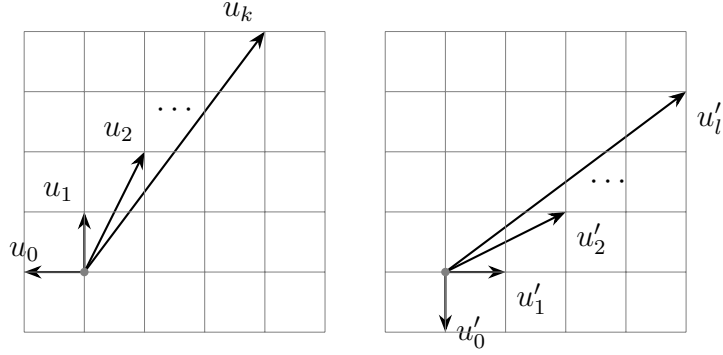

    Because we have know that the curves cannot touch the orbifold point or \(D_0 \cap D_0'\), the curve count is the sum of the corresponding counts in two noncompact toric surfaces, namely those with the divisors $D$ and the divisors $D'$, respectively. The corresponding fans of these two toric surfaces are illustrated in Figure~\ref{fig:two toric surfaces}.
    These counts have been determined (by expressing the Kuranishi structures on the relevant moduli in terms of the corresponding structures for moduli of closed curves) in \cite[Theorem 4.1]{LLW}.  We combine the contributions from the two noncompact surfaces to arrive at the formula:
    \begin{equation} \label{superpotential}
        W = z_1 \prod_{i = 1}^{k} \left( 1 + \frac{1}{z_1 z_2}  \prod_{a = 1}^{i} q_a\right) + z_2 \prod_{j = 1}^{l} \left( 1 + \frac{1}{z_1 z_2} \prod_{b = 1}^{j} q'_b \right).
    \end{equation}
    The meaning of the formula is that each monomial after expansion is a contribution of 
    a holomorphic disk, where the $q_i, q'_i$ variables record the homology contributions of $D_i$ and $D_i'$, and the powers of $z_1, z_2$ record the homology class  of the boundary in $\bT_0$. 

    This determines the numerical contributions of all the disks; we have already remarked that the corresponding boundaries are the $X_i, Y_j$.  It remains only to compare Equations \eqref{superpotential} and \eqref{eq: recursion relation}.  Ours differs by changing $z_2 \to -z_2$, which is a consequence of a different spin structure choice.  One can check that \(\alpha_i\) corresponds to \(\prod_{a=1}^i q_a\) and \(\beta_j\) corresponds to \(\prod_{b = 1}^j q_b'\) by a direct homology computation. 

\end{proof}

\begin{remark}
    Let us recall the crepant resolution conjecture \cite{Ruan-crepant, Bryan-Graber-crepant}: given two crepant resolutions of a given singularity, the Gromov-Witten invariants should agree, possibly up to some change of variable.  An open variant is proposed in \cite{Brini-Cavalieri-Ross-crepant}. 
    
    Let us observe that in the setting of Lemma \ref{why recursion}, 
    if two Calabi-Yau 3-folds $X, X'$ are resolutions of some space with compact singularity locus, we  have tautologically $A_{\partial X, \partial L}(c) = A_{\partial X', \partial L}(c)$, since $(\partial X, \partial L) = (\partial X', \partial L')$.  Thus, when \eqref{boundary equation} determines $Z_{X, L}$, we will also have $Z_{X, L} = Z_{X', L}$. 
    (The identification  $H^2(X) \to  H^2(\partial X) \leftarrow H^2(X')$ will however induce a change of variable in formulas.)

    The CY3 strips we study are all resolutions of the corresponding singular toric CY3 whose fan consists of a single cone (on the trapezoid whose subdivision is dual to the FTCY).  Thus Lemma \ref{reeb positivity}, Theorem \ref{thm:counting at infinity}, and Theorem \ref{solving the recursion} together verify the crepant resolution invariance of the all genus, skein-valued curve counts in this setup.  
\end{remark}

\subsection{A different Lagrangian filling}\label{sec:the_other_brane}

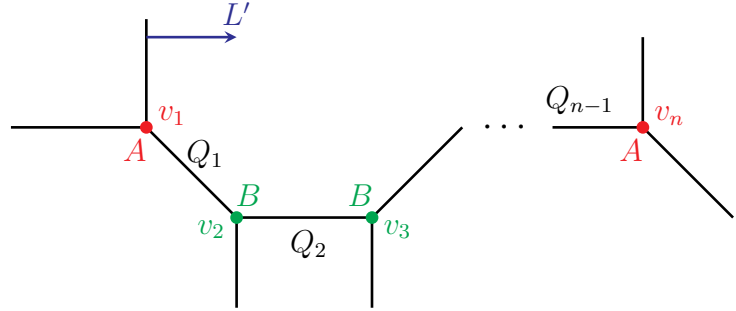
\begin{figure}
    \centering
 
\begin{tikzpicture}[every path/.style={line width=1pt}, xscale = 1.2, yscale = 1.2]
    
    \draw (-1.5, 0) -- (0, 0) -- (0, 1.2);
    \draw (0,0) -- (1, -1); 
    \draw (1, -1) -- (1, -1); 
    \draw (1, -1) -- (1, -2);
    \draw (1, -1) -- (2.5, -1);
    \draw (2.5, -1) -- (2.5, -2);
    \draw (2.5, -1) -- (3.5, 0);

    \draw[Blue, -stealth] (0, 1) -- (1, 1);
    
    \node[above] at (1, 1) {\textcolor{Blue}{\(L'\)}};
       
    \node at (4, 0) {\large $\cdots$} ;

    \draw (4.5, 0) -- (5.5, 0) -- (6.5, -1);
    \draw (5.5, 0) -- (5.5, 1);

     \node (A)  at (0, 0) {};
    \fill[Red] (A) circle (2pt);
    \node [anchor = north] at (A.west) {\color{Red} \(A\) };
    \node [anchor = west] at (A.north) {\color{Red} \(v_1\)};
    \node (B) at (1, -1) {};
    \fill[Green] (B) circle (2pt);
    \node [anchor = south] at (B.east) { \color{Green} \(B\)}; 
    \node [anchor = east] at (B.south) {\color{Green} \(v_2\)};
    \node (B1) at (2.5, -1) {};
    \fill[Green] (B1) circle (2pt);
    \node [anchor = south] at (B1.west) {\color{Green} \(B\)}; 
    \node [anchor = west] at (B1.south) {\color{Green} \(v_3\)};

    \node (An) at (5.5, 0) {};
    \fill[Red] (An) circle (2pt);
    \node [anchor = north] at (An.west) {\color{Red} \(A\)}; 
    \node [anchor = west] at (An.north) {\color{Red} \(v_n\)};

    \node at (.65, -.3) {$Q_1$};
    \node at (1.8, -1.3) {\(Q_2\)};
    \node at (4.8, .3) {\(Q_{n-1}\)}; 
    
    \end{tikzpicture}

    \caption{A different Aganagic--Vafa brane}
    \label{fig:different brane}
\end{figure}
\begin{figure}
\centering
\begin{tikzpicture}[every path/.style={line width=1.1pt}, rotate=270, scale = .7]
 \draw[black] (0,0) rectangle (6,6);

  \draw[Red, mid arrow = .85] 
    (3,3) .. controls (3,4) and (2.8,5) .. (2.7,6);
  \draw[Red, mid arrow=0.75] 
    (2.7,0) -- (2.1,6);
  \draw[Red, mid arrow=. 75] 
    (2.1,0) -- (1.5,6);
  \draw[Red, mid arrow=.75] 
    (.7,0) .. controls (.5,2) and (0, 5) .. (0, 6);
  \node at (1,4) {\large \color{Red} \(\cdots\)};
  \node[right] at (2,6.1) {\color{Red} \(Y'_l\)};

\begin{scope}[shift={(6,0)}, xscale=-1]
     \draw[Green, mid arrow = .85] 
    (3,3) .. controls (3,4) and (2.8,5) .. (2.7,6);
  \draw[Green, mid arrow=0.75] 
    (2.7,0) -- (2.1,6);
  \draw[Green, mid arrow=. 75] 
    (2.1,0) -- (1.5,6);
  \draw[Green, mid arrow=.75] 
    (.7,0) .. controls (.5,2) and (0, 5) .. (0, 6);
  \node at (1,4) {\large \color{Green} \(\cdots\)};
  \node[right] at (2,6.1) {\color{Green} \(X'_k\)};
\end{scope}



  \node[draw=blue, circle, fill = blue, inner sep=1.5pt] at (0, 0) {};
  \node[draw=blue,  circle, fill = blue, inner sep=1.5pt] at (3, 3) {};

 \node[left] at (0, 0) {\color{blue} \large \(c_-\)};
  \node[left] at (3, 3) {\color{blue} \large\(c_+\)};

\end{tikzpicture}
    \caption{The torus \(\bT\) as the boundary of \(L'\)}
     \label{fig:boundary of L'}
\end{figure}
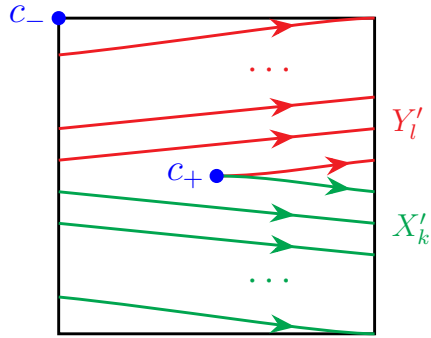

We can also consider the Aganagic--Vafa brane \(L'\) placed on another outer leg, as shown in Figure~\ref{fig:different brane}. This is the brane studied in \cite{IKP, PS19, Banerjee-Hock}.

The brane \(L'\) is an asymptotic Lagrangian filling of the torus-at-infinity \(\bT\). Hence the curve counting problem at infinity is the same.
We continue to adopt the convention that the \((1,0)\) circle is filled. Under this convention, the boundaries of the \(J\)-holomorphic disks---denoted by \(X'_k\) and \(Y'_l\)---are rotated by \(90^\circ\), as illustrated in Figure~\ref{fig:boundary of L'}.

\begin{corollary}\label{thm:recursion of L'}
    Let \(Z_{L'}\) be skein-valued counting of \(L'\). Then the skein-valued counting for \(L'\) satisfies the skein equation: 
    \begin{equation}\label{eq:skein_recursion_L'}
        \left( \sum_{i \ge 0} A_i X'_i - \sum_{j \ge 0} B_j Y'_j.  \right) \cdot Z_{X, L'} = 0. 
    \end{equation}
\end{corollary}
\qed

One has corresponding assertions  for the other fillings.  However, by contrast with Theorem~\ref{solving the recursion}, we do not know finite multiple cover formulas for the counts for \(L'\) or other fillings.


\subsection{Beyond strips?} \label{sec: why strips}
If the FTCY has a loop, or equivalently, the triangulation defining \(X\) has a interior integral point, it is believed that the relationship between the partition function and quantum should be significantly more involved, and may involve new `nonperturbative' invariants \cite{Grassi-Hatsuda-Marino, Codesido-Grassi-Marino, Kashaev-Marino, Rella-resurgence, Rella-Fantini, Francois-Grassi}.

Let us explain from the point of view of Lemma \ref{why recursion} why this should be the case, and what (from this viewpoint) the new invariants must be. 
The point is that the FTCY has a loop if and only if the ideal boundary  \(\partial_\infty X\) has an index zero Reeb orbit. Thus Lemma \ref{why recursion} no longer applies.  Instead, the curve count at infinity gives a relation between the count of compact curves and all counts of curves with possible asymptotic ends at the Reeb orbit.  

One may correspondingly expect that after some elimination theory, the quantum curve will annihilate some combination of the count of compact curves with the counts of curves asymptotic to the Reeb orbit.  (Thus these must be the new invariants.  It is also reasonable that these invariants would be non-perturbative in the K\"ahler parameters, since these curves would in a naive sense have infinite energy.)

  \begin{figure}
    \centering
     
    \begin{subfigure}[t]{.49\linewidth}
    \centering
    \begin{tikzpicture}[line width = 1pt, scale = 1.5]
            \draw (0, 0) -- (1, 0) -- (0, -1) -- (-1, 1) -- (0, 0) -- (0, -1);
            \draw (-1, 1) -- (1, 0);

            \draw[Purple] (.3, -.3) -- (.3, .5) -- (-.5, -.3) -- cycle ;
            \draw[Purple] (.3, -.3) -- (1, -1);
            \draw[Purple] (.3, .5) -- (.8, 1.5);
            \draw[Purple] (-.5, -.3) -- (-1.5, -.8); 
            
        \end{tikzpicture}
        \caption{The FTCY(\textcolor{Purple}{purple}) and the dual triangulation}
    \end{subfigure}
    \begin{subfigure}[t]{.49\linewidth}
        \centering
         \includegraphics[scale = .5]{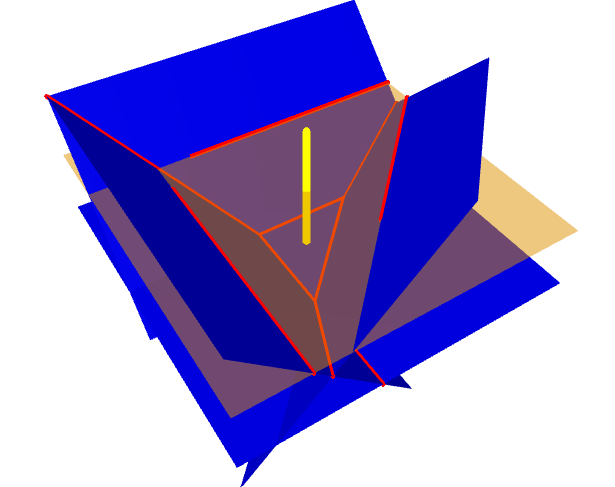}
         \caption{The moment map. The yellow ray is the image of a \(J\)-holomorphic disk.}
         \label{subfig:moment_cone_of_KP2}
    \end{subfigure}
       
        \caption{The toric data of \(K_{\bP^2}\)}
        \label{fig:toric_KP2}
    \end{figure}
    
    \begin{figure}
        \centering
        \begin{tikzpicture}[line width = 1pt]
            \draw (0,0) -- (4, 0) -- (4, 4) -- (0, 4) -- cycle;
            
            \draw[Green, mid arrow] (2, 2) -- (4, 2);
            \draw[Green, mid arrow] (0,2) -- (2,2); 
            \draw[Green, mid arrow] (2, 2) -- (2, 4);
            \draw[Green, mid arrow] (2, 0) -- (2, 2); 
            \draw[Green, mid arrow] (2, 2) -- (0, 0);
            \draw[Green, mid arrow] (4, 4) -- (2, 2);
            \node[right] at (4, 2) {\color{Green} \(P_{1,0}\) };
            \node[below left] at (2, 4) {\color{Green} \(P_{0, 1}\) };
            \node[above left] at (0,0) { \color{Green} \(P_{-1, -1}\)};

            \node[draw=blue, circle, fill = blue, inner sep=1.5pt] at (2, 2){};
        \end{tikzpicture}
        \caption{Boundaries of \(J\)-holomorphic disks in \(\bT = \partial_\infty L_{AV}\subset K_{\bP^2}\)}
        \label{fig:curves_at_infinity_KP2}
    \end{figure}
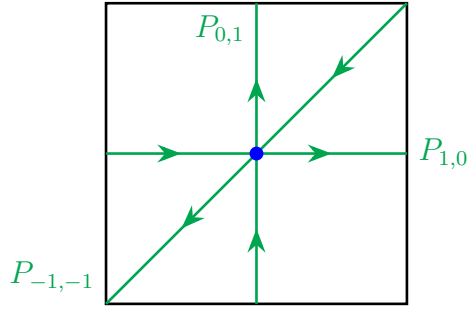

We do not (yet) know how to perform this elimination and rearrangement.  But let us at least illustrate the geometry in the case of local $\mathbb{P}^2$.

    The toric data  is given in Figure~\ref{fig:toric_KP2}. We take the vector \(w = (0 ,0, 1)\). Then the ideal boundary is a contact manifold, isomorphic to the lens space:
    \[
    \partial_\infty X = f_w^{-1}(R) \cong S^5/(\bZ/3).
    \]
    There is a rigid \(J\)-holomorphic curve, lying over the the vertical ray illustrated in yellow in Figure~\ref{subfig:moment_cone_of_KP2}.  Its boundary is an index zero Reeb orbit.

    If we take an Aganagic-Vafa brane \(L_{AV}\), the Legendrian torus \(\bT = \partial_\infty L_{AV}\) still has only positive index Reeb chords.  The rigid disks with one positive puncture can be counted after projection along the symplectic reduction
    \[
     S^5/(\bZ/3) \longrightarrow \bC\bP^2,
    \]
    where their images will end on the standard Clifford torus in \(\bC\bP^2\).  There are three such embedded holomorphic disks, and their boundaries are 
    \(P_{0,1}\), \(P_{1, 0} \), and \(P_{-1, -1}\), as
    drawn in Figure~\ref{fig:curves_at_infinity_KP2}.
    These skein elements have homology classes matching the nonconstant terms in the equation for the mirror curve of \(K_{\bP^2}\):
    \[
        1 +  x + y + \frac{Q}{xy},
    \]


\section{Topological vertex calculations}

In this section, we determine the two-leg topological vertex partition function of a strip, and show that it can be organized via a multiple cover formula. 

\subsection{The topological vertex}\label{subsec:top_vertex}

We write $\Lambda := \mathbb{Z}[x_1, x_2, \cdots]^{\fS_\infty}$ for the ring of symmetric functions; recall it has a linear basis given by the Schur functions $s_\lambda$.  We will often extend scalars or complete the ring of symmetric functions without changing notation. 

Recall the \emph{skew Schur polynomials} \(s_{\lambda/\mu}\), defined by
\[
\langle s_{\lambda/\mu}, s_{\nu} \rangle = \langle s_{\lambda} , s_\mu s_\nu \rangle
\]
where \(\langle \ , \ \rangle\) is the standard inner product for symmetric polynomials.
Let
\begin{align*}
\kappa_\lambda &= \sum_i \lambda_i (\lambda_i - 2i + 1),\\
\rho &= (q^{-1/2}, q^{-3/2}, q^{-5/2}, \dots).
\end{align*}
Note that \(\kappa_\lambda = - \kappa_{\lambda^t}\). 

By definition \cite{AKMV}, the topological vertex is (after some rearrangement by \cite{Okounkov-Reshetikhin-Vafa}): 

\begin{equation}
    \label{the vertex}
    C_{\mu_1 \mu_2 \mu_3}
    := q^{\,\kappa_{\mu_3}/2}\,
   s_{\mu_2}(q^\rho)
   \sum_{\eta}
   s_{\mu_1/\eta}\!\bigl(q^{\,\mu_2^t + \rho}\bigr)
   s_{\mu_3^t/\eta}\!\bigl(q^{\,\mu_2 + \rho}\bigr).
\end{equation}
One introduces also the following `framing' corrections: 
\begin{equation}
    \label{eq:vertex_with_framing}
    C_{\mu_1 \mu_2 \mu_3}^{(f_1, f_2, f_3)}
    := q^{\,f_1 \kappa_{\mu_1}/2 + f_2 \kappa_{\mu_2}/2 + \kappa_{\mu_3}/2}\,
   C_{\mu_1 \mu_2 \mu_3}.
\end{equation}

One collects these into the following element of $\Lambda^{\otimes 3}$: 

\begin{align*}
Z^{(f_1, f_2, f_3)} 
  = \sum_{\mu_1, \mu_2, \mu_3}
    C_{\mu_1 \mu_2 \mu_3}^{(f_1, f_2, f_3)}\,
    s_{\mu_1} \otimes s_{\mu_2} \otimes s_{\mu_3}.
\end{align*}



The building block of the topological vertex is \(\bC^3\) with three branes, and the FTCY graph is illustrated in Figure~\ref{fig:FTCY of C3}. 
An Aganagic–Vafa brane attached to \(e_i\) with framing \(f\) is represented by the vector
\[
e_{i+1} - f e_i
\]
called a \emph{leg}. Here the subscript is understood modulo~\(3\).

\begin{figure}
\centering
\begin{tikzpicture}[line width = 1pt]
    \draw[-Stealth] (0, 0) -- (-1.5, 0);
    \draw[-Stealth]  (0, 0) -- (0, 1.5); 
    \draw[-Stealth] (0, 0) -- (1.5, -1.5);

    \node[right] at (0, 1.5) {\(e_3\)};
    \node[right] at (1.5, -1.5) {\(e_1\)};
    \node[left] at (-1.5, 0) {\(e_2\)};
 
    \draw[Blue, -Stealth] (-.8, 0) -- (-.8, 1.5); 
    \draw[Blue, -Stealth] (1, -1) -- (1, -2.5);
    \draw[Blue, -Stealth] (0, .8) -- (1.5, .8);
\end{tikzpicture}

\caption{The FTCY of \(\bC^3\). The three legs (drawn in blue) attached to \(e_1\), \(e_2\), and \(e_3\) are of framing~\(-1\), \(0\), and \(-1\), respectively.}
\label{fig:FTCY of C3}

\end{figure}
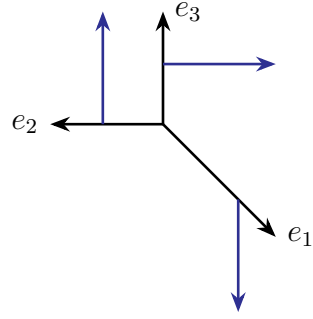

More generally, given any FTCY $X$ with some collection $\mathcal{L}$ of framed Aganagic-Vafa branes, there is a partition function
$$Z_{X, \mathcal{L}} \in \mathbb{Q}(q^{\pm 1/2})[Q_1^\pm, \dots, Q_{n-1}^\pm] \otimes \Lambda^{\otimes \mathcal{L}}.$$ 
It is defined by the following gluing rules.  Decompose the FTCY into vertices, and at each edge place Aganagic-Vafa branes with opposing framings,  as illustrated in Figure~\ref{fig:gluing FTCYs}.
Then evaluate the diagram per the `Feynman rules' where the vertex is the topological vertex, and the propagator for an edge with K\"ahler class $Q$ is: 
\begin{equation}\sum_\lambda  (-1)^{f f' |\lambda|}  Q^{|\lambda|} s_{\lambda}^* \otimes
        s_{\lambda^t}^*
\end{equation}


\begin{figure}
\centering
\begin{tikzpicture}[
    scale = 1,
    every path/.style = {line width = 1pt}
]
    \draw (0,0) -- (-1,1);
    \draw (0,0) -- (1,0);
    \draw (0,0) -- (0,-1);
    \draw[Blue, -Stealth] (1, 0) -- (1, -1);
    \draw[dashed] (1, 0) -- (1.5, 0); 
    \node[below] at (1, -1) {\textcolor{Blue}{\(L\)}}; 
    \node at (-1, -1) {\(X_1\)};
    \begin{scope}[xshift = 2.5cm]
        \draw (0,0) -- (-1,0);
        \draw (0,0) -- (0,1);
        \draw (0,0) -- (1,-1);
        \draw[Blue, -Stealth] (-1, 0) -- (-1, 1); 
    \draw[->] (2, 0) -- (3,0); 
    \node[above] at (-1, 1) {\textcolor{Blue}{\(L'\)}};
    \node at (0, -1) {\(X_2\)};
    \end{scope}

    \begin{scope}[xshift = 8cm]
        \draw (0,0) -- (-1,1);
        \draw (0,0) -- (2,0);
        \draw (0,0) -- (0,-1);
        \draw (2,0) -- (2,1);
        \draw (2,0) -- (3,-1);
        \node[above] at (1, 0) {\(Q\)}; 
    \end{scope}
\end{tikzpicture}
\caption{Gluing vertices.}
\label{fig:gluing FTCYs}
\end{figure}
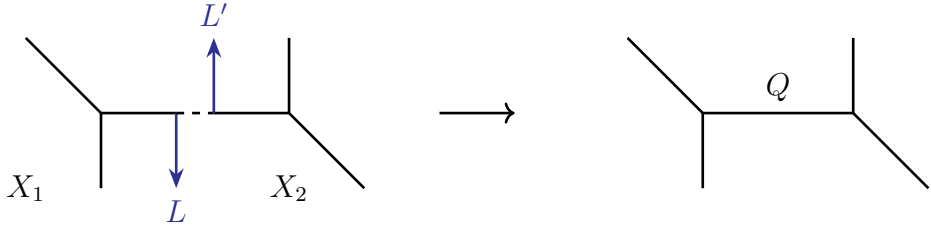

We denote the `open part' by:
\begin{equation*}
    Z^{\mathrm{open}}_{X, L} = Z_{X, L}/Z_{X, \emptyset}. 
\end{equation*}
Note that the degree \(0\) term of \(Z^{\mathrm{open}}\) is always \(1\). 

\subsection{Plethystic exponentials}

 For a \(\lambda\)-ring, denote the plethystic exponentials by 
\begin{align*}
    \Exp{x} &= \exp{\sum_k \frac{1}{k} \psi_k (x)}\\
    \Expp{x} &= \exp{\sum_k \frac{(-1)^{k+1}}{k} \psi_k(x) } 
\end{align*}
where \(\psi_k\) is the \(k\)-th Adams operator. Note that the Plethystic  exponentials satisfy:
\begin{align*}
    \Exp{x + y} &= \Exp{x} \cdot \Exp{y} \\
    \Expp{x+y} & = \Expp{x} \cdot \Expp{y}
\end{align*}
for all elements \(x\) and \(y\) of the \(\lambda\)-ring. 

We will apply this to the $\lambda$-ring $\Lambda$ of symmetric functions.  The Adams operators act by 
$\psi_k(p_n) = p_{kn}$, where $p_k$ are the power sum symmetric functions $p_k := \sum_i x_i^k$.

We later extend scalars to  \(\bQ(q^{\pm 1/2})[[Q_1, \cdots, Q_n]] \), and take \(Q_i\) and \(q^{1/2}\) to be line elements. We also denote the quantum integers by \(\{n\} = q^{n/2}- q^{-n/2}\).

Since we have chosen the framings of \(L_1\) and \(L_2\), their skein modules are canonically identified with \(\Lambda\). 
For a line element \(\xi \in \Lambda\), the skein dilogarithm introduced in Definition~\ref{def : skein dilog} can be rewritten as
\begin{equation}
    \Psi[\xi] = \Exp{\frac{-\xi}{ \{1\}} p_1}
\end{equation}
and we also have
\begin{align*}
    \Psi[\xi]^{-1} &= \Exp{\frac{\xi}{ \{1\}} p_1},   
    &\Psi[-\xi] &= \Expp{ \frac{\xi}{ \{1\} } p_1},  
    &\Psi[-\xi]^{-1} &= \Expp{  \frac{-\xi}{ \{1\}} p_1}.
\end{align*}

We define the following elements of $\Sk(L_1), \Sk(L_2), \Sk(L_1 \sqcup L_2)$

\begin{align}
    \Psi_{\mathrm{disk}}^{ v_k - L_1} &=
    \begin{dcases}
    \Exp{ \frac{\alpha_k}{ \{1\}} p_1^{L_1}},  &\text{if \(v_k\) is of type \(A\)}.  \\
    \Exp{ \frac{-\beta_k}{ \{1\}} p_1^{L_1}},   &\text{if \(v_k\) is of type \(B\)}. 
    \end{dcases}\\
    \Psi_{\mathrm{disk}}^{ v_k - L_2} &=
    \begin{dcases}
    \Expp{  \frac{ Q_{k , n}}{ \{1\}} p_1^{L_2}} , &\text{if \(v_k\) is of type \(A\) and \(v_n\) is of type \(A\).}   \\
        \Expp{  \frac{-Q_{k , n}}{\{1\}} p_1^{L_2}} , &\text{if \(v_k\) is of type \(B\) and \(v_n\) is of type \(A\).}\\
    \Exp{  \frac{-Q_{k , n}}{\{1\}} p_{1}^{L_2}} , &\text{if \(v_k\) is of type \(A\) and \(v_n\) is of type \(B\).}\\
   \Exp{ \frac{Q_{k , n}}{\{1\}} p_1^{L_2}} , &\text{if \(v_k\) is of type \(B\) and \(v_n\) is of type \(B\).}
    \end{dcases} \label{eq:psi v2-L2}\\
    \Psi_{\mathrm{annulus}}^{ L_1 - L_2} &= 
    \begin{dcases}
    \Expp{ Q_{1, n}\, p_1^{L_1} \otimes p_1^{L_2}}, &\text{if \(v_n\) is of type \(A\)}.  \\
    \Exp{ -Q_{1 , n} \,  p^{L_1}_1 \otimes p_1^{L_2}}, & \text{if \(v_n\) is of type \(B\).}
    \end{dcases}
\end{align}
Here the superscript of \(p_1\) indicates which brane it corresponds to.

\subsection{Multiple cover formulas for strip partition functions}

Recall that \(s_\lambda\) forms an \(R\)-linear basis of \(\Lambda\). Denote the dual basis by~\( \{s_\lambda^*\}\). 
Recall the Cauchy identities: 

\begin{align}
    \sum_{\lambda} s_{\lambda} (x) s_{\lambda}(y) &= \prod_{i, j} (1-x_i y_j)^{-1} = \exp{\sum_k \frac{ (-1)^{k+1}}{k} p_k(x) p_k(y)} \label{eq:Cauchy identity 1}\\
    \sum_{\lambda} s_{\lambda} (x) s_{\lambda^t}(y) &= \prod_{i, j} (1+x_i y_j) = \exp{\sum_k \frac{1}{k} p_k(x) p_k(y)} \label{eq:Cauchy identity 2}
\end{align}

One can derive from these the following formulas: 

\begin{lemma}[The gluing formulas]
Let \( \{A_k\}_{k=1}^{\infty}\) and \(\{B_k\}_{k=1}^{\infty}\) be sequences of elements in~\(\Lambda_R\). The following holds (assuming everything is well-defined in the completion of \(\Lambda\)):
\begin{align}
    \left(1 \otimes s_\lambda^* \otimes s_\lambda^* \otimes 1\right) & \left( \exp{\sum_k \frac{1}{k} A_k \otimes p_k} \otimes \exp{\sum_k \frac{1}{k} p_k \otimes B_k}\right) \notag \\
    &= \exp{\sum_k \frac{1}{k}A_k \otimes B_k }, \label{eq:gluing formula 1}\\[10pt]
    \left(\sum 1 \otimes s_\lambda^* \otimes s_{\lambda^t}^* \otimes 1 \right) &\left(\exp{\sum_k \frac{1}{k} A_k \otimes p_k}\otimes \exp{ \sum_k \frac{(-1)^{k+1}}{k} p_k \otimes B_k}\right)\notag \\
    & = \exp{\sum_k \frac{1}{k} A_k \otimes B_k}. \label{eq:gluing formula 2}
\end{align}
\end{lemma}
\begin{proof}
A detailed proof for \eqref{eq:gluing formula 1} can be found in, \cite[Page 22]{Nak24}. The argument of \eqref{eq:gluing formula 2} is similar. 
\end{proof}

Now let us focus on strips. Start with \(\bC^3\) with two legs and framing \((-1, 0)\), and denote the specialization of the topological vertex by:
\begin{align*}
    Z^{(-1, 0)} = \sum_{\mu_1, \mu_2} C_{\mu_1 \mu_2 \emptyset}^{(-1, 0, 0)} \, s_{\mu_1}(x) s_{\mu_2}(y)
\end{align*}

\begin{proposition} \label{prop:GV for two legs}
    We have
    \begin{equation*}
        Z^{(-1, 0)} = \Expp{\frac{1}{\{1\}} p_1(x) } \Exp{\frac{1}{\{1\} } p_1(y)} \Expp{ p_1(x)p_1(y)}
    \end{equation*}
\end{proposition}

\begin{proof}
First, by Cauchy identities~\eqref{eq:Cauchy identity 1}\eqref{eq:Cauchy identity 2} we have
\begin{align*}
    \sum_\lambda s_\lambda(q^{\rho}) s_\lambda(x) &= \Exp{ \frac{1}{\{1\}} p_1(x)}\\
    \sum_\lambda s_\lambda(q^{\rho}) s_{\lambda^t}(x) &= \Expp{ \frac{1}{\{1\}} p_1(x)}\\
\end{align*}
Recall the following generalized Cauchy identities (\cite[p.94]{Mac98}):
\begin{align*}
    \sum_\eta s_{\eta/\lambda} (x) s_{\eta/\mu}(y) &= \prod_{i,j} (1-x_i y_j)^{-1} \sum_{\tau} s_{\mu/\tau} (x) s_{\lambda/\tau} (y)\\
    \sum_\eta s_{\eta/\lambda^t}(x) s_{\eta^t/\mu}(y) &= \prod_{i,j}(1+x_i y_j) \sum_{\tau} s_{\mu^t/\tau} (x) s_{\lambda/\tau^t} (y).
\end{align*}
In particular, 
\begin{align}
    \sum_\eta s_{\eta/\mu}(x) s_\eta(y) 
    &= \prod_{i, j \ge 1} (1-x_iy_j )^{-1} \cdot s_\mu(y) \notag\\
    &= \sum_{\nu} s_\nu(x) s_\nu (y) s_\mu(y) \label{eq:generalized_Cauchy_1}\\
    \sum_\eta s_{\eta/\mu}(x) s_{\eta^t}(y) &= \prod_{i, j\ge 1} (1+ x_i y_j)  s_{\mu^t}(y) \notag\\
    &= \sum_{\nu} s_\nu(x) s_{\nu^t} (y) s_{\mu^t}(y) \label{eq:generalized_Cauchy_2}
\end{align}

By ~\eqref{the vertex} and \eqref{eq:vertex_with_framing}, we have:
\begin{align*}
C_{\mu, \nu, \emptyset}^{-1, 0, 0} &=  (-1)^{|\mu| + |\nu|} q^{\frac{\kappa_\mu }{2} } \sum_\eta s_{\mu/\eta} (q^{-\rho} ) s_{\nu^t/\eta} (q^{-\rho})
\end{align*}
and we have the standard identity
\begin{equation}\label{eq:s_rho_transpose}
  s_{\lambda/\mu}(q^\rho) = (-1)^{|\lambda| - |\mu|} s_{\lambda^t/\mu^t} (q^{-\rho}).
\end{equation}

Thus, 
\begin{align*}
    Z^{(-1,0)} &= \sum_{\mu, \nu} (-1)^{|\mu| + |\nu|}  \sum_\eta s_{\mu/\eta} (q^{-\rho} ) s_{\nu^t/\eta} (q^{-\rho}) s_\mu(x) s_\nu (y)\\
    &=\sum_{\mu, \eta} (-1)^{|\mu|}  s_{\mu/\eta}(q^{-\rho}) s_{\mu} (x) \sum_{\nu} (-1)^{|\nu|} s_{\nu^t/\eta}(q^{-\rho}) s_\nu(y)\\
    &\stackrel{\eqref{eq:s_rho_transpose}}{=} \sum_{\mu, \eta} (-1)^{|\mu|}  s_{\mu/\eta}(q^{-\rho}) s_{\mu} (x) (-1)^{|\eta|}  \sum_{\nu} s_{\nu/\eta^t}(q^{\rho}) s_\nu(y)\\
    &\stackrel{\eqref{eq:generalized_Cauchy_2}}{=}\sum_{\mu, \eta} (-1)^{|\mu|}  s_{\mu/\eta}(q^{-\rho}) s_{\mu} (x) (-1)^{|\eta|} s_{\eta^t}(y)  \sum_{\nu} s_{\nu}(q^\rho) s_{\nu}(y)  \\
    &\stackrel{\eqref{eq:s_rho_transpose}}{=} \sum_{\mu, \eta} s_{\mu^t/\eta^t}(q^\rho) s_{\mu}(x) s_{\eta^t}(y) \Exp{\frac{1}{\{1\} } p_1(y)} \\
    &= \sum_{\mu, \eta} s_{\mu^t} (q^\rho) s_{\mu}(x) s_{\eta}(x) s_{\eta^t} (y) \Exp{\frac{1}{\{1\} } p_1(y)} \\
    &= \Expp{\frac{1}{\{1\}} p_1(x) } \Expp{ p_1(x)p_1(y)} \Exp{\frac{1}{\{1\} } p_1(y)} 
\end{align*}
\end{proof}

\begin{theorem} \label{thm: topological vertex for strips}
For a strip $X$ with (our choice of) two legs $L_1, L_2$:
\begin{equation}\label{eq:partition for two branes}
    Z^{\mathrm{open}}_{X, L_1 \cup L_2}  = \Psi_{\mathrm{annulus}}^{L_1 - L_2}  \cdot \prod_{k = 1}^n \Psi_{\mathrm{disk}}^{v_k - L_1} \prod_{k=1}^n \Psi_{\mathrm{disk}}^{v_k - L_2} 
\end{equation}
\end{theorem}

\begin{proof}
   We prove \eqref{eq:partition for two branes} by induction. 
 Proposition~\ref{prop:GV for two legs} is the case of \(1\) vertex. Suppose \eqref{eq:partition for two branes} is true for any strips with \(n\) vertices. Let us consider a strip with \(n+1\) vertices \(X\), which can be obtained by gluing a \(\bC^3\) to a strip with \(n\) vertex \(X'\), as in Figure~\ref{fig:induction on strips}. 
    
    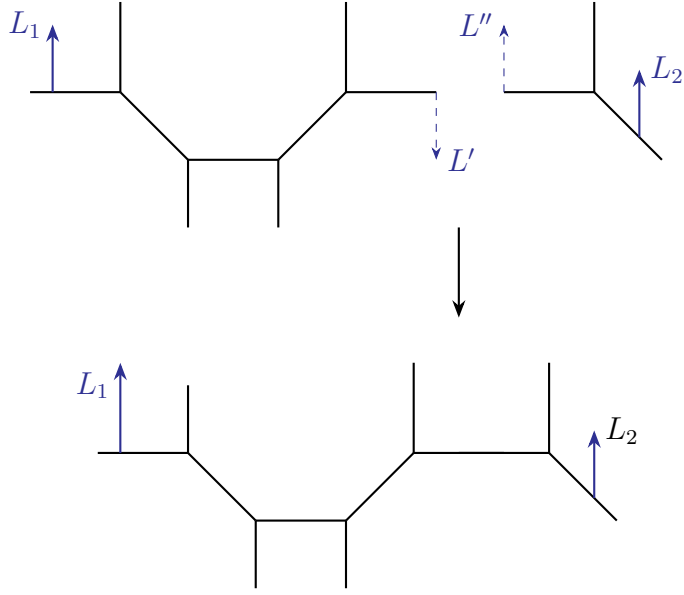
\begin{figure}
        \centering
        \begin{tikzpicture}[scale = .6]
        \begin{scope}[yscale = -1, shift = { (-1.5, 0)}]
            \draw[thick] (0,0) -- (2, 0);
            \draw[thick] (0,0) -- (-1.5, 1.5);
            \draw[thick] (0,0) -- (0, -2);
            \draw[dashed, Blue, -Stealth] (2, 0) -- (2, 1.5); 
            \draw[thick] (-1.5, 1.5) -- (-3.5, 1.5);
            \draw[thick] (-1.5, 1.5) -- (-1.5, 3);
            \draw[thick] (-3.5, 1.5) -- (-3.5, 3);
            \draw[thick] (-3.5, 1.5) -- (-5, 0); 
            \draw[thick] (-5, 0) -- (-7, 0);
            \draw[thick] (-5, 0) -- (-5, -2); 
            \draw[thick, Blue, -Stealth] (-6.5, 0) -- (-6.5, -1.5);
            \node[left] at (-6.5, -1.5) {\color{Blue} \(L_1\) } ;
            
            \node[right] at (2, 1.5) {\color{Blue} \(L'\) } ;

            \begin{scope}[shift = {(5.5, 0)}]
            \draw[thick] (0, 0) -- (-2, 0); 
            \draw[thick] (0, 0) -- (0, -2);
            \draw[thick] (0, 0) -- (1.5, 1.5); 
            \draw[dashed, Blue, -Stealth] (-2, 0) -- (-2, -1.5);
            \node[left] at (-2, -1.5) {\color{Blue} \(L''\) } ;

            \draw[thick, Blue, -Stealth] (1, 1) -- (1, -.5);
            \node[right] at (1, -.5) {\color{Blue}\(L_2\)};

            \end{scope}
        \end{scope}

            \draw[thick, -Stealth] (1, -3) -- (1, -5);

            \begin{scope}[shift = {(0, -8)}, yscale = -1]
                   \draw[thick] (0,0) -- (2, 0);
            \draw[thick] (0,0) -- (-1.5, 1.5);
            \draw[thick] (0,0) -- (0, -2);
            
            \draw[thick] (-1.5, 1.5) -- (-3.5, 1.5);
            \draw[thick] (-1.5, 1.5) -- (-1.5, 3);
            \draw[thick] (-3.5, 1.5) -- (-3.5, 3);
            \draw[thick] (-3.5, 1.5) -- (-5, 0); 

            \draw[thick] (-5, 0) -- (-7, 0);
            \draw[thick] (-5, 0) -- (-5, -1.5); 
            \draw[thick, Blue, -Stealth] (-6.5, 0) -- (-6.5, -2);
            \node[left] at (-6.5, -1.5) {\color{Blue} \(L_1\) } ;

            \begin{scope}[shift = {(3, 0)}]
            \draw[thick] (0, 0) -- (-2, 0); 
            \draw[thick] (0, 0) -- (0, -2);
            \draw[thick] (0, 0) -- (1.5, 1.5); 
            \draw[thick, Blue, -Stealth] (1, 1) -- (1, -.5);
            \node[right] at (1, -.5) {\(L_2\)};

            \end{scope}
            \end{scope}
        \end{tikzpicture}
        \caption{Adding a new vertex to a strip}
        \label{fig:induction on strips}
    \end{figure}

    We compute the case when the vertex of adjacent to \(L'\) the vertex adjacent to \(L''\) are both of type \(A\). The other (three) cases are similar. 
By the gluing rule for the topological vertex and the induction hypothesis, we have
\begin{align}
    Z^{\mathrm{open}}_{X, L_1 \cup L_2}  &= \left( \sum_\alpha 1 \otimes s_{\alpha}^* \otimes  Q^{|\alpha|}   s_{\alpha^t}^* \otimes 1 \right) \cdot \left(
    Z^{\mathrm{open}}_{X', L_1 \cup L'}
    \otimes Z^{(-1,0)} \right)  \label{eq:the partion function after gluing}\\
     & = \left( \sum_\alpha 1 \otimes s_{\alpha}^* \otimes  Q^{|\alpha|}   s_{\alpha^t}^* \otimes 1 \right) \cdot \notag \\
     &    \left( Z_{\emptyset}\Psi_{\mathrm{annulus}}^{L_1 - L'}  \cdot \prod_{k = 1}^n \Psi_{\mathrm{disk}}^{v_k - L_1} \prod_{k=1}^n \Psi_{\mathrm{disk}}^{v_k - L'} \otimes 
    \Exp{\frac{p_1^{L''}}{\{1\}}} \Expp{ p_1^{L''} \otimes p_1^{L_2}  } \Psi_{\mathrm{disk}}^{v_{n+1} - L_2} 
    \right) 
    \notag
\end{align}
where \(Z_\emptyset\) is some polynomial living in \(\bQ(q^{\pm 1/2})[[Q_1, \dots, Q_{n-1}]]\), as defined before. 
Note that the only terms that will change under the gluing are those involves \(L'\) and \(L''\).
Hence we consider:
\begin{align}
    &\left( 
        \sum_\alpha 1 \otimes s_{\alpha}^* \otimes Q^{|\alpha|} s_{\alpha^t}^* \otimes 1 
    \right)
    \cdot 
    \left( 
        \psi_{\mathrm{annulus}}^{L_1 - L'} 
        \prod_{k=1}^n \Psi_{\mathrm{disk}}^{v_k - L'} 
        \otimes 
        \Exp{\frac{P^{L''}_{0,1}}{\{1\}}} 
        \Expp{ p_1^{L''} \otimes p_1^{L_2} }  
    \right) 
    \notag \\
    &= \left( 
        \sum_\alpha 1 \otimes s_{\alpha}^* \otimes Q^{|\alpha|} s_{\alpha^t}^* \otimes 1
    \right) 
    \cdot 
    \notag \\
    &\exp{
        \sum_l \frac{1}{l}
        \left(  
            (-1)^{l+1} Q_{1,n}^l p_1^{L_1} 
            + \sum_k \sum_l \pm \frac{1}{\{l\}} Q_{k,n}^l 
        \right) 
        \otimes p_l^{L'} 
    } 
    \otimes  
    \notag \\
    &\exp{
        \sum_l \frac{1}{l} \, p_l^{L''} \otimes 
        \left(
            \frac{1}{\{l\}} + (-1)^{l+1} p_l^{L_2}  
        \right)
    } 
    \notag \\
    &= (\text{use the gluing formula \eqref{eq:gluing formula 2}}) 
    \notag \\
    &\exp{ 
        \sum_l \frac{Q_{n,n+1}^l}{l}   
        \left(  
            Q_{1,n}^l\, p_l^{L_1} 
            + \sum_k \sum_l \pm \frac{1}{l \{l\}} Q_{k,n}^l 
        \right)
        \otimes 
        \left(
            \frac{1}{\{l\}} + (-1)^{l+1} p_l^{L_2}  
        \right)
    } 
    \notag \\
    &= C \, \exp{
        \sum_l \frac{1}{l} 
        \left(
            \frac{1}{\{l\}} Q_{1,n+1}^l \, p_l^{L_1} 
            + (-1)^{l+1} Q_{1,n+1}^l\, p_l^{L_1} \otimes p_l^{L_2}  
            + \sum_k \pm \frac{1}{\{l\}} Q_{k,n+1}^l\, p_l^{L_2}
        \right)
    } 
    \notag \\
    &= C \, 
    \Psi_{\mathrm{annulus}}^{L_1 - L_2} \, 
    \Psi_{\mathrm{disk}}^{v_{n+1} - L_1} \, 
    \prod_{k=1}^{n+1} \Psi_{\mathrm{disk}}^{v_k - L_2}. 
    \notag
\end{align}
Here $C \in \bQ(q^{\pm 1/2})[[Q_1^\pm, \cdots, Q_n^\pm]]$ is the contribution of closed curves, obtained by gluing the disks from \(v_k\) to \(L'\) and the disks from  \(v_{n+1}\) to \(L''\). The sign of the term~\(  Q_{k, n+1}^l \, p_l^{L_2} \) depends of the type of vertex \(v_k\), following the rule in \eqref{eq:psi v2-L2}.
In fact, this term does not matter much for our purpose, but one can check the signs do match up. 

    In conclusion, we get 
    \begin{align*}
    Z^{\mathrm{open}}_{X, L_1 \cup L_2}& =\Psi_{\mathrm{annulus} }^{L_1 - L_2} \prod_{k = 1}^{n+1}\Psi_{\mathrm{disk}}^{v_k - L_1} \prod_{k=1}^{n+1} \Psi_{\mathrm{disk}}^{v_k - L_2} 
    \end{align*}
\end{proof}

\begin{remark}
    Comparing Theorems \ref{solving the recursion}, \ref{thm:counting at infinity} on the one hand with Theorem \ref{thm: topological vertex for strips} on the other that  for CY3 strips with (our choice of) one Aganagic-Vafa brane, the skein-valued curve count agrees with the formula derived from the topological vertex. (We note this does not yet follow formally from \cite{ELS}, because the necessary gluing formula has not yet been established.)
\end{remark}

\bibliographystyle{plain}
\bibliography{refs}

\end{document}

%% file: macros.tex
\newcommand{\bC}{\mathbb{C}}

\newcommand{\bL}{\mathbb{L}}

\newcommand{\bP}{\mathbb{P}}
\newcommand{\bQ}{\mathbb{Q}}
\newcommand{\bR}{\mathbb{R}}
\newcommand{\bZ}{\mathbb{Z}}
\newcommand{\bT}{\mathbb{T}}

\newcommand{\bfT}{\mathbf{T}}

\newcommand{\fg}{\mathfrak{g}}

\DeclareMathOperator{\Hom}{Hom}

\newcommand{\R}{\mathbb{R}}

\newcommand{\Ad}{\mathrm{Ad}}

\newcommand{\Sk}{\mathrm{Sk}}

\newcommand{\fS}{\mathfrak{S}}

\newcommand{\Exp}[1]{
\mathcal{E}\left(#1\right) 
}

\newcommand{\Expp}[1]{
\mathcal{E}' \left(#1\right) 
}

\renewcommand{\exp}[1]{
\mathrm{exp}\left(#1\right)
} 

\newtheorem{dummy}{dummy}[section]
\newtheorem{lemma}[dummy]{Lemma}
\newtheorem{theorem}[dummy]{Theorem}

\newtheorem{corollary}[dummy]{Corollary}
\newtheorem{proposition}[dummy]{Proposition}

\theoremstyle{definition}
\newtheorem{definition}[dummy]{Definition}

\newtheorem{example}[dummy]{Example}
\newtheorem{remark}[dummy]{Remark}